\documentclass[12pt]{article}
\usepackage{amsfonts}
\usepackage{latexsym,amsfonts,amsmath,amsthm,makeidx}
\usepackage{mathrsfs}
 \usepackage{color}
\usepackage{makeidx}
\usepackage{geometry}
\makeindex

\newtheorem{theorem}{Theorem}[section]
\newtheorem{lemma}{Lemma}[section]

\newtheorem{proposition}{Proposition}[section]
\newtheorem{remark}{Remark}[section]

\numberwithin{equation}{section}

\begin{document}
\title{\bf Fast and slow decaying solutions for $H^{1}$-supercritical quasilinear Schr\"{o}dinger equations\thanks{Supported by NSFC (No.11371146) and NSERC of Canada;
E-mail: chengyk@scut.edu.cn; jcwei@math.ubc.ca.}
\date{}
\author{
{\bf Yongkuan Cheng$^{1}$,\ \ Juncheng Wei$^{2}\thanks{Corresponding author.}$
}\\ {\small\it $^1$School of Mathematics, South China University of Technology,}
\\{\small\it Guangzhou, 510640, PR China}
\\ {\small\it $^2$Department of Mathematics, University of British Columbia,}
\\{\small\it Vancouver, B. C., V6T 1Z2, Canada}
}}
\maketitle
\begin{center}{\it   }\end{center}
\vskip0.36in

{\bf Abstract }
 We consider the following quasilinear Schr\"{o}dinger equations of the form
 \begin{equation*}
\triangle u-\varepsilon V(x)u+u\triangle u^2+u^{p}=0,\ u>0\ \mbox{in}\ \mathbb{R}^N\ \mbox{and}\ \underset{|x|\rightarrow \infty}{\lim} u(x)=0,
\end{equation*}
where $N\geq 3,$ $p>\frac{N+2}{N-2},$ $\varepsilon>0$ and $V(x)$ is a positive function. By imposing appropriate conditions on $V(x),$ we prove that, for $\varepsilon=1,$ the existence of infinity many positive solutions with slow decaying $O(|x|^{-\frac{2}{p-1}})$ at infinity if $p>\frac{N+2}{N-2}$ and, for $\varepsilon$ sufficiently small, a positive solution with fast decaying $O(|x|^{2-N})$ if $\frac{N+2}{N-2}<p<\frac{3N+2}{N-2}.$ The proofs are based on perturbative approach. To this aim, we also analyze the structure of positive solutions for the zero mass problem.

{\small{\bf Keywords}\quad
Nonlinear Schr\"{o}dinger equations;
$H^{1}$-supercritical; Fast and slow decaying solutions} \\

{\small{\bf  MSC}\quad 	35J20;
35J60; 35Q55} \vspace{0.5cm}

\vskip0.6in
\section{Introduction}

The nonlinear Schr\"{o}dinger equation
\begin{equation}\label{original-equation}
iz_t=-\triangle z+W(x)z-[\triangle |z|^2]z-|z|^{p-1}z,\ (t,x)\in(0,\infty)\times \mathbb{R}^{N},
\end{equation}
where $W:\mathbb{R}^{N}\rightarrow \mathbb{R}$ is a given potential,
has been introduced in \cite{Brizhik-2001,Brizhik-2003,Hart-2003} to study a model of a self-trapped electrons in quadratic or hexagonal lattices (see also \cite{Brihaye-2004}). In those references numerical and analytical results have been given.

Here of particular  interest is in the existence of standing wave solutions, that is, solutions of type $z(x,t)=\exp (-iE t)u(x),$ where $E\in\mathbb{R}.$ Assuming that the amplitude $u(x)$ is positive and vanishing at infinity, it is well known that $z$ satisfies (\ref{original-equation}) if and only if the function $u$ solves the following equation of quasilinear elliptic type
\begin{equation}\label{WZQ-equation}
\begin{cases}
\triangle u-\varepsilon V(x)u+u\triangle u^2+u^{p}=0,\ x\in\mathbb{R}^N;\\
 u>0\  \mbox{and}\ \underset{|x|\rightarrow \infty}{\lim} u(x)=0,
\end{cases}
\end{equation}
where $V(x)=W(x)-E$ is the new potential function. In the rest of this paper we will assume that $V(x)$ is a bound and positive function.

Because of the presence of the quasilinear term $u\triangle u^2,$ we can see that $p=\frac{3N+2}{N-2}$ is the critical exponent for the existence of solutions from the view of variational matheods.  For the subcritical case, that is, $1<p<\frac{3N+2}{N-2},$ construction of solutions to this problem by variational methods has been a hot topic during the last decade. A typical result for the equation (\ref{WZQ-equation}) is, up to our knowledge, due to Liu, Wang and Wang \cite{Liu-2003}. The idea in \cite{Liu-2003} is to make a change of variable and reduce the quasilinear problem (\ref{WZQ-equation}) to a semilinear one and the Orlicz space framework is used to prove the existence of positive solutions via the mountain pass theorem. Subsequently, the same method of changing of variable is also used in Colin and Jeanjean \cite{Colin-2004}, but the usual Sobolev space $H^{1}(\mathbb{R}^{N})$ is used as the working space. Recently, Shen and Wang in \cite{Shen-2013} study the following generalized quasilinear Schr\"{o}dinger equation:
\begin{equation}\label{Shen-equation}
-\mbox{div}(g^{2}(u)\nabla u)+g(u)g'(u)|\nabla u|^2+V(x)u=h(u),\ x\in\mathbb{R}^{N},
\end{equation}
where $g^{2}(s)=1+\frac{1}{2}(l(s^2)')^{2}.$ By introducing the variable replacement
\begin{equation}\label{Shen-replacement}
v=G(u)=\int_{0}^{u}g(t)\mbox{d}t,\ u=G^{-1}(v)
\end{equation}
and imposing some conditions on $V(x),$ the authors obtain the positive solution for (\ref{Shen-equation}) with a general function $l(s)$ when $h(s)$ is superlinear and subcritical. But under the condition
\begin{equation}\label{V}
\underset{|x|\rightarrow\infty}{\lim}|x|^{2}V(x)=0,
\end{equation}
the solvability of the equation (\ref{WZQ-equation}) with $1<p<\frac{3N+2}{N-2}$ still remains open.


Subcriticality is a rather essential constraint in the use of many variational methods devised in the literature and many papers \cite{Poppenberg-Wang-2002, Wang-Proc-2002, Wang-JDE-2004, Moameni-JDE-2006, Moameni-Non-2006} focused on the subcritical case. Very little is known in the supercritical case since a major technical obstacle in understanding such problems stems from the lack of Sobolev embeddings suitably fit to a weak formulation of this problem. Direct tools of the calculus of variation, very useful in subcritical, and even critical cases, are not appropriate in the supercritical. In the critical case, Liu et al. in \cite{Liu-2003} asked the following open question: are there solutions for (\ref{WZQ-equation}) in the case of $p=\frac{3N+2}{N-2}?$
However, generally speaking, except some results relate to the critical exponent, see, for instance, \cite{Soares-JDE-2010, Wang-CVPDE-2013, Wang-JDE-2013, He-Non-2013, Ye-DCDS-2016, Silv-2010, Deng-2016, Shen-2016-cpaa}, there are still no conclusive results about the existence of positive solutions for the problem (\ref{WZQ-equation}) with $p=\frac{3N+2}{N-2}$ or $ p>\frac{3N+2}{N-2}$.

In all the papers mentioned above variational methods are used.  In this paper,
we shall explore the distinctive nature of this problem for having {\em two} critical exponents, one being $p=\frac{3N+2}{N-2}$ (from the quasilinear term $u \Delta u^2$) and the other being $p=\frac{N+2}{N-2}$ which is $H^1$-critical (from the term $\Delta u$). We shall concentrate in the problem (\ref{WZQ-equation}) when the exponent $p$ is $H^{1}$-supercritical, that is, $p>\frac{N+2}{N-2},$ (which includes $p=\frac{3N+2}{N-2}$), and we establish a new phenomenon from the viewpoint of {\em singular perturbations}.
Noticing that (\ref{WZQ-equation}) is a quasilinear problem, we adopt the change of variables which enable us to convert the original quasilinear problem (\ref{WZQ-equation}) into a semilinear problem
\begin{equation}\label{WZQ-equation-change}
\begin{cases}
\triangle v-\varepsilon V(x)\frac{G^{-1}(v)}{g(G^{-1}(v))}+f(v)=0,\ x\in\mathbb{R}^{N};\\
v>0\  \mbox{and}\ \underset{|x|\rightarrow \infty}{\lim} v(x)=0,
\end{cases}
\end{equation}
where $f(v)=\frac{G^{-1}(v)^{p}}{g(G^{-1}(v))}$ and $g(s)=\sqrt{1+2s^{2}}.$ Thus, if $v$ is a solution of (\ref{WZQ-equation-change}), we have $u=G^{-1}(v)$ is a solution of $(\ref{WZQ-equation}).$

A solution $v$ to (\ref{WZQ-equation-change}) is called fast decaying if $ v = O(|x|^{2-N})$ at infinity and slow decaying if $ v >> O(|x|^{2-N})$.
Then, to describe our result about the fast and slow decaying solutions,
our starting point is the zero mass problem
\begin{equation}\label{zero-mass}
\begin{cases}
\triangle u+u\triangle u^2+u^p=0,\ x\in\mathbb{R}^{N};\\
 u>0\  \mbox{and}\ \underset{|x|\rightarrow \infty}{\lim} u(x)=0.
\end{cases}
\end{equation}
Applying the change of variables (\ref{Shen-replacement}) again, the quasilinear problem (\ref{zero-mass}) can be reduced to the equations of the form
\begin{equation}\label{zero-mass-change}
\begin{cases}
\triangle v+f(v)=0,\ x\in\mathbb{R}^{N};\\
 v>0\  \mbox{and}\ \underset{|x|\rightarrow \infty}{\lim} v(x)=0.
\end{cases}
\end{equation}

Our first result concerns with the structure of positive radial solutions of the zero mass problem (\ref{zero-mass}).
\begin{theorem}\label{result-zero-mass}
Suppose that $p>1.$ Then
\begin{itemize}
\item[{\rm (1).}] there exist no fast decaying solutions to the problem (\ref{zero-mass}) if $p\geq \frac{3N+2}{N-2}$ or $1<p\leq\frac{N+2}{N-2};$
\item[{\rm (2).}] there exist a unique fast decaying radial solution to the problem (\ref{zero-mass}) if $\frac{N+2}{N-2}<p<\frac{3N+2}{N-2};$
\item[{\rm (3).}] there exist a one-parameter family of slow decaying radial solutions  to the problem (\ref{zero-mass}) if $p>\frac{N+2}{N-2}.$
\end{itemize}
\end{theorem}

\begin{remark}\label{zero-mass-remark-1}
Some cases of the results of Theorem \ref{result-zero-mass} are contained in \cite{Adachi-2012,Adachi-2016}. More specifically, similarly to the standard Liouville theorem, if $1<p<\frac{N+2}{N-2},$ the authors proved the nonexistence results of fast decay solutions to (\ref{zero-mass}) (See \cite{Adachi-2012}). In \cite{Adachi-2016}, the authors showed the existence of a unique fast decay solution and a one-parameter family of slow decay solutions to (\ref{zero-mass}) if $\frac{N+2}{N-2}<p<\frac{3N+2}{N-2}$ via the results introduced in \cite{Tang-2001}.  Moreover, the authors in \cite{Adachi-2016} also pointed out that they did not know whether there are solutions for the equation (\ref{zero-mass}) with $p=\frac{N+2}{N-2}.$ Particularly, in Theorem \ref{result-zero-mass}, we draw the definite conclusion about this case by using the Pohozeav identity.
\end{remark}
Theorem \ref{result-zero-mass} shows that the structure of solutions  changes along with the variations of the power $p$ and we remark that the solvability of the equation (\ref{zero-mass}) heavily depends on the power $p.$ Let us explain the main reason for such a rich phenomenon. On  one hand, $f(v)\rightarrow v^{p}$ as $v\rightarrow 0.$ On the other hand, $f(v)\rightarrow 2^{\frac{p-3}{4}}v^{\frac{p-1}{2}}$ as $v\rightarrow +\infty.$ That is, the nonlinearity $f$ is not a pure power of $v$ but $f$ has both $H^{1}$-subcritical and $H^{1}$-supercritical growth in $v>0.$ In \cite{Tang-1997}, the authors consider a similar model
$$
f(u)=\begin{cases}u^{p}\ \ \ u\geq 1;\\
u^{q}\ \ \ u<1,\end{cases}
$$
where $1<p<\frac{N+2}{N-2}<q$ and give an almost complete description for the structure of positive radial solutions by a shooting argument.

The following result is about the fast decaying solutions of the equation (\ref{WZQ-equation}).
\begin{theorem}\label{result-2}
Assume that
\begin{equation}\label{V-condition}
V>0,\ V\in L^{\infty}(\mathbb{R}^{N})\ \mbox{and}\ V(x)=o(|x|^{-2})\ \mbox{as}\ |x|\rightarrow +\infty
\end{equation}
hold.
Then for $\varepsilon$ sufficiently small
the problem (\ref{WZQ-equation}) has a positive fast decaying solution if $\frac{N+2}{N-2}<p<\frac{3N+2}{N-2}.$
\end{theorem}

Compared with Theorem \ref{result-zero-mass}, it is natural to ask whether the nonexistence of a fast decaying solution remains true for (\ref{WZQ-equation}) when $p\geq \frac{3N+2}{N-2}.$ This may be in general a difficult question to answer if no other conditions imposed on $V(x).$ For the special case $x\cdot \nabla V(x)+2V(x)\geq 0,$ the authors in \cite{Severo-2017} show the nonexistence results of fast decay solutions by a Pohozeav identity for the equation (\ref{WZQ-equation}) in the case $p\geq \frac{3N+2}{N-2}$ and $\varepsilon=1.$

Our final result concerns the existence of slow decaying solutions.

\begin{theorem}\label{result-3}
Assume that $\varepsilon=1.$ Then
the problem (\ref{WZQ-equation}) has a continuum of solutions $u_{\lambda}(x)$ such that $\underset{\lambda\rightarrow 0}{\lim}u_{\lambda}(x)=0$ uniformly in $\mathbb{R}^{N}$ either
$N\geq 4,$ $p>\frac{N+1}{N-3}$ and the condition (\ref{V-condition}) holds or $N\geq 3,$ $\frac{N+2}{N-2}<p< \frac{N+1}{N-3}$ and there exist $C>0,$ $\mu>N$ such that
\begin{equation}\label{slow-V-condition-2}
V(x)\leq C|x|^{-\mu}\ \mbox{for}\  x\in\mathbb{R}^{N}.
\end{equation}
\end{theorem}
\begin{remark}
In this theorem, we answer the question raised in \cite{Liu-2003} for $p=\frac{3N+2}{N-2}.$
\end{remark}

The proofs of Theorems \ref{result-2} and \ref{result-3} are  based perturbative approach,  introduced by Davila, del Pino, Musso and  Wei   \cite{Pino-Wei-2007,Pino-2008,Pino-Wei-2008,Wang-Wei-2012} in the study of
fast and slow decaying solutions for second order or nonlinear Schr\"{o}dinger equations and exterior domain problems. Some of our ideas are motivated from these papers.

In the fast-decaying case, we consider the problem (\ref{WZQ-equation-change}) as small perturbation of the problem $(\ref{zero-mass-change})$ when $\varepsilon>0$ is sufficiently small. For a point $\xi\in\mathbb{R}^{N}$ used as the reference origin, the function $v_{f}(x+\xi)$ is considered as an initial approximation, where $v_{f}$ is a solution of (\ref{zero-mass-change}). This function will constitute a good approximation for small $\varepsilon.$ By adjusting $\xi,$ we prove that the solutions we want can be achieved.

As for the slow decay solution of the equation (\ref{WZQ-equation}), we set $\varepsilon=1$ and consider the equation with a parameter $\lambda$ by means of replacing the variable $v$ in the equation (\ref{WZQ-equation-change}) by $\lambda^{\frac{2}{p-1}}v(\lambda x+\xi)$
\begin{equation}\label{WZQ-equation-change-lambda}
\begin{cases}
\triangle v-V_{\lambda}(x)\lambda^{-\frac{2}{p-1}}\frac{G^{-1}(\lambda^{\frac{2}{p-1}}v)}{g(G^{-1}(\lambda^{\frac{2}{p-1}}v))}+\lambda^{-\frac{2p}{p-1}}f(\lambda^{\frac{2}{p-1}}v)=0,\ x\in\mathbb{R}^{N};
\\ v>0\  \mbox{and}\ \underset{|x|\rightarrow \infty}{\lim} v(x)=0,
\end{cases}
\end{equation}
where $\lambda>0,$ $\xi\in\mathbb{R}^{N}$ and $V_{\lambda}(x)=\lambda^{-2}V(\frac{x-\xi}{\lambda}).$
 We observe that $\lambda^{-\frac{2}{p-1}}\frac{G^{-1}(\lambda^{\frac{2}{p-1}}v)}{g(G^{-1}(\lambda^{\frac{2}{p-1}}v))}\rightarrow v$ and $\lambda^{-\frac{2p}{p-1}}f(\lambda^{\frac{2}{p-1}}v)\rightarrow v^{p}$ as $\lambda\rightarrow 0.$ Thus the problem may be regarded as small perturbation of the problem
$$
\triangle v-V_{\lambda}v+v^{p}=0
$$
when $\lambda>0$ is sufficiently small. Consequently, infinitely many positive solutions with slow decay $O(|x|^{-\frac{2}{p-1}})$ at infinity can be constructed similar to the perturbative procedure introduced by  Davila, del Pino, Musso and Wei \cite{Pino-Wei-2007}.

In this paper, we make use of the following notations: the symbol $C$ denotes a positive constant (possibly different) independent with $\lambda.$ $A\sim B$ if and only if there exist two positive constants $a,b$ such that $aA\leq B\leq bA.$ $ v_f$ denotes the unique fast decaying solution of (\ref{zero-mass-change}).

\section{Proof of Theorem \ref{result-zero-mass}}

In this section, we analyze the structure of positive decaying solutions (\ref{zero-mass}). We first prove the nonexistence of fast-decaying solutions for $p\leq \frac{N+2}{N-2}$ or $p\geq \frac{3N+2}{N-2}$ by using the Pohozaev identity. Then we show the existence of fast decaying solution for (\ref{zero-mass}) by using the classical Berestycki-Lions condition in \cite{Berestycki-Lions} for $\frac{N+2}{N-2}<p<\frac{3N+2}{N-2}.$ Finally we use a perturbative approach to prove the existence of a family of slow-decaying solutions for $p>\frac{N+2}{N-2}.$

To prove the nonexistence results for the equation (\ref{zero-mass-change}), we recall the following Pohozaev identity.

\begin{lemma}\label{Pohozaev-1} {\rm (Pohozaev identity)} Suppose $F(x,u,r)\in C^{1}(\mathbb{R}^{N}\times \mathbb{R}\times \mathbb{R}^{N})$ satisfies
\begin{equation}\label{Pohozaev-equation}
{\rm div} F_{r}(x,u,\nabla u)=F_{u}(x,u,\nabla u),
\end{equation}
where
$$
F_{r}(x,u,r)=(F_{r_{1}}(x,u,r),F_{r_{2}}(x,u,r),\cdots,F_{r_{N}}(x,u,r)),\ r=(r_{1},r_{2},\cdots,r_{N}),
$$
$$
F_{r_{i}}(x,u,r)=\frac{\partial F(x,u,r)}{\partial r_{i}},\ i=1,2,\cdots,N
$$
and
$$
F_{u}(x,u,r)=\frac{\partial F(x,u,r)}{\partial u}.
$$
Then, if $F(x,u,\nabla u),$ $x\cdot F_{x}(x,u,\nabla u)$ and $F_{r}(x,u,\nabla u)\cdot\nabla u\in L^{1}(\mathbb{R}^{N}),$ there holds the following identity
\begin{equation}\label{Pohozaev-identity}
N\int_{\mathbb{R}^{N}}F(x,u,\nabla u)\mbox{d}x+\int_{\mathbb{R}^{N}}x\cdot F_{x}(x,u,\nabla u)\mbox{d}x-\int_{\mathbb{R}^{N}}F_{r}(x,u,\nabla u)\cdot \nabla u\mbox{d}x=0.
\end{equation}
\end{lemma}
We omit the proof of this lemma, since it can be mainly found in \cite{Pucci-1986}.

To present the Pohozaev identity associated to (\ref{zero-mass}), we rewrite the equation (\ref{zero-mass}) as
\begin{equation}\label{our-equation-Pohozaev}
\mbox{div}\left( g^2(u)\nabla u\right)-g(u)g'(u)|\nabla u|^2+ u^{p}=0.
\end{equation}
Thus, the integrands in (\ref{Pohozaev-identity}) can be expressed as
$$
F(x,u,\nabla u)=\frac{1}{2}g^2(u) |\nabla u|^2-\frac{1}{p+1}u^{p+1},
$$
$$
x\cdot F_{x}(x,u,\nabla u)=0
$$
and
$$
F_{r}(x,u,\nabla u)\cdot \nabla u=g^2(u) |\nabla u|^2.
$$
Consequently, we achieve the following lemma based on Lemma \ref{Pohozaev-1} under the conditions
$|\nabla u|^2,$ $u^2|\nabla u|^2$ and $u^{p+1}\in L^{1}(\mathbb{R}^{N}).$
\begin{lemma}\label{Pohozaev-our-equation}
Suppose that $u\in C^{2}(\mathbb{R}^{N})$ is a solution of (\ref{zero-mass}). Then
\begin{equation}\label{Pohozaev-identity-our-equation}
\begin{split}
\frac{N-2}{2}\int_{\mathbb{R}^{N}}\left(1+2 u^2\right)|\nabla u|^2\mbox{d}x=\frac{N}{p+1}\int_{\mathbb{R}^{N}}|u|^{p+1}\mbox{d}x
\end{split}
\end{equation}
if $|\nabla u|^2,$ $u^2|\nabla u|^2$ and $u^{p+1}\in L^{1}(\mathbb{R}^{N}).$
\end{lemma}

Equations (\ref{zero-mass}) can be rewritten as
\begin{equation}\label{zero-mass-g}
-\mbox{div}(g^2(u)\nabla u)+g(u)g'(u)|\nabla u|^2=u^{p},\ x\in\mathbb{R}^{N},
\end{equation}
where $g^2(u)=1+2u^2.$ By Lemma \ref{Pohozaev-our-equation}, the Pohozaev identity associated to (\ref{zero-mass-g}) is
\begin{equation}\label{P-identity}
\frac{N-2}{2}\int_{\mathbb{R}^{N}}|\nabla u|^2\mbox{d}x+(N-2)\int_{\mathbb{R}^{N}}u^2|\nabla u|^2\mbox{d}x=\frac{N}{p+1}\int_{\mathbb{R}^{N}}u^{p+1}\mbox{d}x.
\end{equation}
On the other hand, the classical solution $u\in D^{1,2}(\mathbb{R}^{N})$ of (\ref{zero-mass-g}) satisfies
$$
\int_{\mathbb{R}^{N}}[g^2(u)\nabla u \nabla \phi+g(u)g'(u)|\nabla u|^2\phi]\mbox{d}x=\int_{\mathbb{R}^{N}}u^p\phi\mbox{d}x.
$$
By taking $\phi=u,$ we achieve
\begin{equation}\label{zero-mass-solution}
\int_{\mathbb{R}^{N}}|\nabla u|^2\mbox{d}x+4\int_{\mathbb{R}^{N}}u^2|\nabla u|^2\mbox{d}x=\int_{\mathbb{R}^{N}}u^{p+1}\mbox{d}x.
\end{equation}
Consequently, combining (\ref{P-identity}) and (\ref{zero-mass-solution}), we have
\begin{equation}\label{P-and-solution}
\left[\frac{N-2}{2}-\frac{N}{p+1}\right]\int_{\mathbb{R}^{N}}|\nabla u|^{2}\mbox{d}x+\left[(N-2)-\frac{4N}{p+1}\right]\int_{\mathbb{R}^{N}}u^2|\nabla u|^2\mbox{d}x=0.
\end{equation}
If $p\geq \frac{3N+2}{N-2},$ then $(N-2)-\frac{4N}{p+1}\geq 0$ and $\frac{N-2}{2}-\frac{N}{p+1}>0.$ Therefore, (\ref{P-and-solution}) implies that $u=0$ under this situation.
Similarly, if $p\leq \frac{N+2}{N-2},$ it follows that $(N-2)-\frac{4N}{p+1}< 0$ and $\frac{N-2}{2}-\frac{N}{p+1}\leq 0.$ Thus, (\ref{P-and-solution}) also shows that $u=0.$
So there are no nonzero solutions for (\ref{zero-mass}) if $p\leq \frac{N+2}{N-2}$ or $p\geq \frac{3N+2}{N-2}.$

This proves (1) of Theorem \ref{result-zero-mass}.

Next we prove the existence of fast decaying solutions to (\ref{zero-mass-change}). By the change of variable   $u=G^{-1}(v)$ we only need to consider  (\ref{zero-mass-change}). To this end we recall the following classical proposition  by Berestycki and Lions \cite{Berestycki-Lions}.
\begin{proposition}\label{Prop-B-L}
Suppose that the following assumptions hold:
\begin{itemize}
\item[{\rm (F-1).}] $f(0)=0$ and $ \underset{s\rightarrow 0^{+}}{\overline{\lim}}\frac{f(s)}{s^{l}}\leq 0,$ where $l=\frac{N+2}{N-2};$
\item[{\rm (F-2).}] There exists $\zeta>0$ such that $F(\zeta)>0,$ where $F(\zeta)=\int_{0}^{\zeta}f(s)\mbox{d}s;$
\item[{\rm (F-3).}] Let $\zeta_{0}=\inf\left\{\zeta: \zeta>0,\ F(\zeta)>0\right\}.$ If $f(s)>0$ for all $s>\zeta_{0},$ then $\underset{s\rightarrow +\infty}{\lim}\frac{f(s)}{s^{l}}=0.$
\end{itemize}
Then the problem (\ref{zero-mass-change}) has a positive, spherically symmetric and decreasing (with $r$) solution $v$ such that $v\in D^{1,2}(\mathbb{R}^{N})\cap C^{2}(\mathbb{R}^{N}).$
\end{proposition}

We now  show that $f(s)$ satisfies the conditions (F-1)-(F-3) in Proposition \ref{Prop-B-L}.

By the definition of $f(s),$ we know that (F-2) is trivial. Noticing that $\underset{s\rightarrow 0}{\lim}\frac{G^{-1}(s)}{s}=1,$ we have
$$
\underset{s\rightarrow 0^{+}}{\lim}\frac{f(s)}{s^{l}}=\underset{s\rightarrow 0^{+}}{\lim}\frac{G^{-1}(s)^{p}}{g(G^{-1}(s))s^{l}}=\underset{s\rightarrow 0^{+}}{\lim}\frac{s^{p}}{s^{l}}=0,
$$
which shows that $f(s)$ satisfies the condition (F-1).

To verify the condition (F-3), it suffices to show that
$$
\underset{s\rightarrow +\infty}{\lim}\frac{f(s)}{s^{l}}=0
$$
since $\zeta_{0}=0$ and $f(s)>0$ for all $s>0.$ Combining the fact $\underset{s\rightarrow +\infty}{\lim}\frac{G^{-1}(s)}{\sqrt{s}}=2^{\frac{1}{4}},$ we deduce that
$$
\underset{s\rightarrow +\infty}{\lim}\frac{f(s)}{s^{l}}=\underset{s\rightarrow +\infty}{\lim}\frac{G^{-1}(s)^p}{g(G^{-1}(s))s^{l}}=\underset{s\rightarrow +\infty}{\lim}\frac{2^{\frac{p-3}{4}}s^{\frac{p-1}{2}}}{s^l}=0.
$$

This proves (2) of Theorem \ref{result-zero-mass}.

Finally we prove (3) of Theorem \ref{result-zero-mass}. To prove the existence of slow decay solutions, since we are considering the autonomous case, that is, $V(x)\equiv 0,$ we can restrict to the radially symmetric case. For this reason, we take $v(x)=v(r),$ where $r=|x|.$

We first consider the problem in the entire space
$$
\begin{cases}
\triangle u+u^{p}=0,\ x\in\mathbb{R}^{N};\\
u(0)=1.
\end{cases}
$$
It is well known that this problem possesses a unique positive symmetric solution
$w(|x|)$
whenever $p>\frac{N+2}{N-2}.$ Then all radial solutions to this problem defined in $\mathbb{R}^{N}$ can be expressed as
$$
w_{\lambda}(|x|)=\lambda^{\frac{2}{p-1}}w(\lambda |x|),\ \lambda>0
$$
and, at a main order, one has
$$
w(r)=C_{p,N}r^{-\frac{2}{p-1}}+o(1)\ \mbox{as}\ r=|x|\rightarrow +\infty,
$$
which implies that this behavior is actually common to all solutions $w_{\lambda}(r).$

Since the problem (\ref{zero-mass-change}) does not carry any parameter explicitly, for $\lambda>0,$
we can make parameters appear by means of replacing the variable $v$ in the equation by $\lambda^{\frac{2}{p-1}}v(\lambda |x|),$ in such a way the problem (\ref{zero-mass-change}) becomes
\begin{equation}\label{zero-mass-change-lambda}
\begin{cases}
\triangle v+\lambda^{-\frac{2p}{p-1}}f(\lambda^{\frac{2}{p-1}}v)=0,\ r\in (0,+\infty);
\\ v>0\  \mbox{and}\ \underset{r\rightarrow \infty}{\lim} v(r)=0.
\end{cases}
\end{equation}
 Then, jointly with the properties of $G^{-1}(v)=v+o(1)$ and $g(G^{-1}(v))=1+o(1)$ as $v\rightarrow 0,$ if $v$ is uniformly bounded, we observe that $\lambda^{-\frac{2p}{p-1}}f(\lambda^{\frac{2}{p-1}}v)\rightarrow v^{p}$ as $\lambda\rightarrow 0.$ Thus the problem may be regarded as small perturbation of the problem
$$
\triangle v+v^{p}=0
$$
when $\lambda>0$ is sufficiently small. Consequently, a positive solution with slow decay $O(|x|^{-\frac{2}{p-1}})$ at infinity can be constructed by asymptotic analysis and Liapunov-Schmidt reduction method.
To be more specific, the idea of the proof of Theorem \ref{result-zero-mass}-(3) is, for $\lambda$ small, to consider the function $\lambda^{\frac{2}{p-1}}w(\lambda |x|)$ as an initial approximation. This scaling will constitute a good approximation under our situations for $\lambda$ sufficiently small. Then, by a classical fixed point argument for contraction mappings, we prove that (\ref{zero-mass-change-lambda}) possesses solutions as desired.
Similar idea has been used in \cite{Pino-Wei-2007,Pino-Wei-2008}.

Under appropriate norms
\begin{equation}\label{fast-norm-1}
\|\phi\|_{*}=\underset{ |x| \leq 1}{\sup} |x|^{\sigma}|\phi(x)|+\underset{ |x|\geq 1}{\sup} |x|^{\frac{2}{p-1}}|\phi(x)|
\end{equation}
and
\begin{equation}\label{fast-norm-2}
\|h\|_{**}=\underset{ |x|\leq 1}{\sup} |x|^{2+\sigma}|h(x)|+\underset{ |x|\geq 1}{\sup} |x|^{2+\frac{2}{p-1}}|h(x)|,
\end{equation}
where $\sigma>0,$ we first consider the solvability of the linear problem
\begin{equation}\label{2-fast-linear-operator}
\begin{cases}\phi ''+\frac{N-1}{r}\phi '+pw^{p-1}\phi=h,\ r\in (0,+\infty);\\
\underset{r\rightarrow +\infty}{\lim}\phi(r)=0
\end{cases}
\end{equation}
and thus we need the following lemma which is Lemma A. 1 proved by D\'avila, del Pino, Musso and Wei \cite{Pino-Wei-2007}.
\begin{lemma}\label{zero-lemma-linear-1}
Assume $0<\sigma<N-2$ and $p> \frac{N+2}{N-2}.$ Then there exists a constant $C>0$ such that for any $h$ satisfying $\|h\|_{**}<+\infty,$ equation (\ref{2-fast-linear-operator}) has a solution $\phi =\mathcal{T}(h)$ such that $\mathcal{T}$ define a linear map and
$$
\|\phi\|_{*}=\|\mathcal{T}(h)\|_{*}\leq C\|h\|_{**}.
$$
\end{lemma}

Let us look for a solution to (\ref{zero-mass-change-lambda}) of the form $v=w+\phi,$ which yields the following equation for $\phi=\phi(r)$
\begin{equation}\label{zero-solve-linear-operator-1}
\begin{cases}\triangle \phi+pw^{p-1}\phi=S(w)+N(\phi),\ r\in (0,+\infty);\\
\underset{r\rightarrow +\infty}{\lim}\phi(r)=0,
\end{cases}
\end{equation}
where
$$S(w)=-\triangle w-\lambda^{-\frac{2p}{p-1}}f(\lambda^{\frac{2}{p-1}}w)$$
and
$$N(\phi)=\lambda^{-\frac{2p}{p-1}}f(\lambda^{\frac{2}{p-1}}w)+pw^{p-1}\phi -\lambda^{-\frac{2p}{p-1}}f(\lambda^{\frac{2}{p-1}}(w+\phi)).
$$

We first estimate the error $\|S(w)\|_{**}$ of the approximate solution.
The fact
\begin{align*}
|S(w)|&
=| \lambda^{-\frac{2p}{p-1}}f(\lambda^{\frac{2}{p-1}}w)-w^{p} |
\end{align*}
and the properties of the change of variables (\ref{Shen-replacement}) show that, for $C_{p}>0,$
\begin{align*}
S(w)=C_{p}\lambda^{\frac{4}{p-1}}w^{p+2}+o\left(\lambda^{\frac{4}{p-1}}w^{p+2}\right)\ \mbox{as}\ \lambda\rightarrow 0.
\end{align*}
Thus, it follows that
$$
| S(w) | \leq C\lambda^{\frac{4}{p-1}} |w|^{p+2}.
$$
We then conclude
\begin{equation}\label{N-2-estimate-1}
\begin{split}
\underset{|x|\leq 1}{\sup}|x |^{2+\sigma} | S(w) |  &\leq C\lambda^{\frac{4}{p-1}}\|w\|_{\infty}^{p+2}\underset{|x |\leq 1}{\sup}|x |^{2+\sigma}
\leq C\lambda^{\frac{4}{p-1}}.
\end{split}
\end{equation}
On the other hand, recalling that
$w(x)\leq C (1+ |x|)^{-\frac{2}{p-1}}\ \mbox{for}\ x\in \mathbb{R}^{N},$
we obtain
\begin{equation}\label{2-N-2-estimate-2}
\begin{split}
\underset{|x|\geq 1}{\sup}|x|^{2+\frac{2}{p-1}} | S(w) |  &\leq C\lambda^{\frac{4}{p-1}}\underset{|x|\geq 1}{\sup}\left(\frac{|x|}{1+|x|}\right)^{\frac{2(p+2)}{p-1}}
\leq C\lambda^{\frac{4}{p-1}}.
\end{split}
\end{equation}
From (\ref{N-2-estimate-1}) and (\ref{2-N-2-estimate-2}), we have
\begin{equation}\label{2-N-2-estimate-3}
\|  S(w)  \|_{**}\leq C\lambda^{\frac{4}{p-1}}.
\end{equation}

In what follows, the proof relies on the contraction mapping theorem. We observe that $\phi$ solves (\ref{zero-solve-linear-operator-1}) if and only if $\phi$ is a fixed point for the operator
$$
\phi=\mathcal{T}(S(w)+N(\phi)),
$$
where $\mathcal{T}$ is introduced in Lemma \ref{zero-lemma-linear-1}. That is to say, $\phi$ solves (\ref{zero-solve-linear-operator-1}) if and only if $\phi$ is a fixed point for the operator
$$
\mathcal{A}(\phi):=\mathcal{T}(S(w)+N(\phi)).
$$
We define
$$
\Sigma =\left\{\phi : \mathbb{R}^{N}\rightarrow \mathbb{R}\  \big |\  \| \phi \|_{*}\leq  C\lambda^{\frac{4}{p-1}}\right\}
$$
and we will prove that $\mathcal{A}$ has a fixed point in $\Sigma.$

For any $\phi\in \Sigma$ and $\sigma\in \left(0,\min\left\{2,\frac{2}{p-1}\right\}\right),$ according to the arguments given in \cite{Pino-Wei-2007}, we have
\begin{equation}\label{zero-mass-N-estimate}
\|N(\phi)\|_{**}\leq C [\|\phi\|_{*}^{2}+\|\phi\|_{*}^{p}]
\end{equation}
 since
$$
N(\phi)=w^p+pw^{p-1}\phi-(w+\phi)^{p}+o(1)\ \mbox{as}\ \lambda \rightarrow 0.
$$
Therefore, combining (\ref{2-N-2-estimate-3}), (\ref{zero-mass-N-estimate}) and Lemma \ref{zero-lemma-linear-1}, it follows that
\begin{equation}\label{zero-mass-in}
\begin{split}
\|\mathcal{A}(\phi)\|_{*} &\leq C [\|S(w)\|_{**}+\|N(\phi)\|_{**}]\\&
\leq C[\lambda^{\frac{4}{p-1}} +\lambda^{\frac{8}{p-1}}+\lambda^{\frac{4p}{p-1}} ]
\leq C \lambda^{\frac{4}{p-1}},
\end{split}
\end{equation}
which implies that $\mathcal{A}(\Sigma)\subset \Sigma.$

We still have to prove that $\mathcal{A}$ is a contraction mapping in $\Sigma.$ Let us take $\phi_{1},\phi_{2}\in \Sigma.$ Then we have
\begin{equation}\label{zero-mass-contraction}
\|\mathcal{A}(\phi_{1})-\mathcal{A}(\phi_{2})\|_{*}\leq C\|N({\phi_{1}})-N(\phi_{2})\|_{**}.
\end{equation}
Moreover, noting that
$$
|N(\phi_{1})-N(\phi_{2})|\leq C\left(w^{p-2}(|\phi_{1}|+|\phi_{2}|)+|\phi_{1}|^{p-1}+|\phi_{2}|^{p-1}\right)|\phi_{1}-\phi_{2}|,
$$
we have the estimate
\begin{align*}
\|\mathcal{A}(\phi_{1})-\mathcal{A}(\phi_{2})\|_{*}& \leq C\|N({\phi_{1}})-N(\phi_{2})\|_{**}\\&
\leq C\left[ \|\phi_{1}\|_{*}^{\min\left\{1,p-1\right\}} +\|\phi_{2}\|_{*}^{\min\left\{1,p-1\right\}}\right] \|\phi_{1}-\phi_{2}\|_{*}\\&
\leq \frac{1}{2}\|\phi_{1}-\phi_{2}\|_{*}
\end{align*}
for suitable small $\lambda.$ This means that $\mathcal{A}$ is a contraction mapping from $\Sigma$ into itself, and hence a fixed point $\phi$ in this region indeed exists. 
So the function $v_{\lambda}(|x|):=\lambda^{\frac{2}{p-1}}(w(\lambda |x|)+\phi(\lambda |x|))$ is a continuum solutions of (\ref{zero-solve-linear-operator-1}) satisfying $\underset{\lambda\rightarrow 0}{\lim} v_{\lambda}(|x|)=0$ uniformly in $\mathbb{R}^{N}$ and $u_{\lambda}(|x|)=G^{-1}(v_{\lambda}(|x|))$ is our desired solution.
This complete the proof of Theorem \ref{result-zero-mass}.

\section{Proof of Theorem \ref{result-2}}

In this section, we will construct a fast decaying solution to the problem (\ref{WZQ-equation}) when $\frac{N+2}{N-2}<p<\frac{3N+2}{N-2}$ by the reduction method. The idea of the proof of Theorem \ref{result-2} is, for $\xi\in\mathbb{R}^{N}$ and $\varepsilon$ small, to consider the function $v_{f}(x+\xi)$ as an initial approximation, where $v_{f}(x)$ is the unique positive radial solution of the zero mass problem (\ref{zero-mass}) stated in Theorem \ref{result-zero-mass}. These functions will constitute good approximations under our situations for suitable $\xi\in\mathbb{R}^{N}$ and $\varepsilon$ sufficiently small. Then, by adjusting $\xi,$ we prove that (\ref{WZQ-equation}) possesses a solution as desired.



At the beginning, we state some notations which will be used in the following. We consider the initial value problem
\begin{equation}\label{fast-initial-problem}
\begin{cases}v''+\frac{N-1}{r}v'+f(v)=0,\ r\in (0,+\infty);\\
v(0)=d>0,\ v'(0)=0,
\end{cases}
\end{equation}
where $f(s)=\frac{G^{-1}(s)^{p}}{g(G^{-1}(s))}.$
By Theorem \ref{result-zero-mass}, there exists a unique $d^{*}>0$ such that the corresponding solution $v_{f}(r;d^{*})$ is the unique positive fast decay solution. Moreover, $z_{0}(r):=\frac{\partial v_{f}}{\partial d}(r;d^{*}) $ satisfies the following initial value problem
\begin{equation}\label{fast-initial-linear-problem}
\begin{cases}\phi''+\frac{N-1}{r}\phi'+f'(v_{f})\phi=0,\ r\in (0,+\infty);\\
\phi(0)=1>0,\ \phi'(0)=0.
\end{cases}
\end{equation}
Then by Lemma 4.4 in \cite{Adachi-2016}, we have that $v_{f}$ is non-degenerate in $D_r^{1,2}(\mathbb{R}^{N})$--radial functions in $ D^{1,2}$. Our next lemma shows that it is nondegenerate in the class of bounded functions. 
Let $Z_{i}=\frac{\partial v_{f}}{\partial x_{i}}$ for $1\leq i\leq N.$  Then we  have the following result.

\begin{lemma}\label{orthogonality}
If $\phi$ satisfies $|\phi|\leq C$ and
\begin{equation}\label{phi-orthogonality}
\triangle \phi+f'(v_{f})\phi=0,\ x\in\mathbb{R}^{N},
\end{equation}
then $\phi\in W=\mbox{Span}\left\{Z_{1},Z_{2},\cdots,Z_{n}\right\}.$
\end{lemma}

\begin{proof}
If $\phi$ is bounded and satisfies (\ref{phi-orthogonality}), by bootstrapping, we achieve $\phi(x)=O(|x|^{2-N})$ as $|x|\rightarrow +\infty.$ Expanding $\phi$ as
$$
\phi(x)=\sum_{k=0}^{\infty}\phi_{k}(r)\Theta_{k}(\vartheta)
$$
we see that $\phi_{k}$ is a solution of
 \begin{equation}\label{fast-linear-ode-2}
\phi_{k}''+\frac{N-1}{r}\phi_{k}'+\left(f'(v_{f})-\frac{\lambda_{k}}{r^2}\right)\phi_{k}=0\ \mbox{for}\ \mbox{all}\ r>0\ \mbox{and}\ k\geq 0.
\end{equation}
For mode $0,$ noticing that $\lambda_{0}=0,$ we know $\phi_{0}(r)$ is a solution of (\ref{fast-linear-ode-2}) and, by
Lemma 4.2 in \cite{Adachi-2016}, $\phi_{0}(r)$ satisfies
$$
r^{\lambda^*} \phi_{0}(r)\rightarrow -\infty\ \mbox{as}\ r\rightarrow \infty,
$$
where $\lambda^*=\begin{cases}\frac{N-1}{2}\ \mbox{if}\ N\geq 4;\\
\frac{1}{2}\ \ \ \ \ \mbox{if}\ N=3.\end{cases}$
Thus, if $\phi_{0}(r)\in D_{rad}^{1,2}(\mathbb{R}^{N}),$ we conclude that
$$
r^{\lambda^*}z_{0}(r)=\begin{cases}
O(r^{\frac{3-N}{2}}),\ \mbox{if}\ N\geq 4;\\
O(r^{-\frac{1}{2}}),\ \mbox{if}\ N=3
\end{cases}
$$
which is a contradiction.
For mode $k$ with $k>1,$ according to Lemma A. 3 in \cite{Pino-Wei-2007}, we conclude that the solution $\phi_{k}$ to (\ref{fast-linear-ode-2}) is zero by the maximum principle. Consequently, jointly with the nondegeneracy in radial class,
we have
$$\phi=\phi_{1}\in\mbox{Span}\left\{Z_{1},Z_{2},\cdots,Z_{N}\right\}.$$\end{proof}


We introduce  appropriate norms
\begin{equation}\label{fast-norm-1}
\|\phi\|_{*,\xi}=\underset{ x\in \mathbb{R}^{N}}{\sup}<x-\xi>^{\sigma}|\phi(x)|
\end{equation}
and
\begin{equation}\label{fast-norm-2}
\|h\|_{**,\xi}=\underset{x\in\mathbb{R}^{N}}{\sup} <x-\xi>^{2+\sigma} |h(x)|,
\end{equation}
where $<\cdot>:=\left(1+|\cdot |^{2}\right)^{\frac{1}{2}}$ and $0<\sigma< N-2.$
We first solve the linear problem
\begin{equation}\label{fast-linear-operator}
\begin{cases}\triangle \phi+f'(v_{f})\phi=h+\sum_{i=1}^{N}c_{i}f'(v_{f})Z_{i},\ x\in\mathbb{R}^{N};\\
\int_{\mathbb{R}^{N}}f'(v_{f})\phi Z_{i}=0,\ i=1,2,\cdots,N;\\
 \underset{|x|\rightarrow +\infty}{\lim}\phi(x)=0.
\end{cases}
\end{equation}

\begin{lemma}\label{fast-lemma-linear-1}
Let $\Lambda >0$ and $|\xi|\leq \Lambda.$ Assume $\frac{N+2}{N-2}<p<\frac{3N+2}{N-2}$ and $\sigma<N-2.$ Then there is a linear map $(\phi,c_{1},\cdots,c_{N})=\mathcal{T}(h)$ defined whenever $\|h\|_{**,\xi}<\infty$ such that $(\phi,c_{1},\cdots,c_{N})$ satisfies (\ref{fast-linear-operator}) and
\begin{equation}\label{fast-linear-norm-up}
\|\phi\|_{*,\xi}+\sum_{i=1}^{N}|c_{i}|\leq C\|h\|_{**,\xi}.
\end{equation}
Moreover, $c_{i}=0$ for all $1 \leq i\leq N$ if and only if
\begin{equation}\label{fast-linear-operator-condition}
\int_{\mathbb{R}}h\frac{\partial v_{f}}{\partial x_{i}}=0\ \mbox{for}\ 1\leq i\leq N.
\end{equation}
\end{lemma}
\begin{proof}
We will divide the proof into two steps.

{\bf Step 1. A priori estimate}

By taking $h=h^{(1)}+h^{(2)}$ in (\ref{fast-linear-operator}), where $h^{(1)}\in W_{1}=\left\{f'(v_{f})Z_{1},f'(v_{f})Z_{2},\cdots,f'(v_{f})Z_{N}\right\}$ and $h^{(2)}\in W_{1}^{\bot},$ we have
\begin{equation}\label{fast-linear-split}
\triangle \phi+f'(v_{f})\phi=h^{(1)}+h^{(2)}+\sum_{i=1}^{N}c_{i}f'(v_{f})Z_{i}.
\end{equation}
If we take $h^{(1)}=-\sum_{i=1}^{N}c_{i}f'(v_{f})Z_{i},$ that is,
\begin{equation}\label{c-i-value}
\ c_{i}=-\frac{\int_{\mathbb{R}^{N}}h^{(1)}Z_{i}}{\int_{\mathbb{R}^{N}}f'(v_{f})|Z_{i}|^2}\ \mbox{for}\ i=1,2,\cdots,N,
\end{equation}
it follows from (\ref{fast-linear-split}) that
\begin{equation}\label{fast-linear-split-1}
\triangle \phi+f'(v_{f})\phi=h^{(2)}
\end{equation}
and $c_{i}=0$ for all $1\leq i\leq N$ if and only if
$$
\int_{\mathbb{R}}h\frac{\partial v_{f}}{\partial x_{i}}=0\ \mbox{for}\ 1\leq i\leq N.
$$
So, in what follows, we consider
\begin{equation}\label{fast-linear-operaror-1}
\begin{cases}
\triangle \phi+f'(v_{f})\phi=h^{(2)},\ x\in\mathbb{R}^{N};\\
\int_{\mathbb{R}^{N}}f'(v_{f})\phi Z_{i}=0,\ i=1,2,\cdots,N;\\
\underset{|x|\rightarrow +\infty}{\lim}\phi(x)=0.
\end{cases}
\end{equation}

We first prove the priori estimates (\ref{fast-linear-norm-up}) by using the contradiction argument. Suppose that there exist $\phi_{n},$ $h^{(2)}_{n}$ such that $\|\phi_{n}\|_{*,\xi}=1$ and $\|h^{(2)}_{n}\|_{**,\xi}=o(1)$ as $n\rightarrow +\infty.$
By the definition of $\|\phi_{n}\|_{*,\xi},$ we can take $x_{n}\in\mathbb{R}^{N}$ with the property
\begin{equation}\label{x-j-norm}
<x_{n}-\xi>^{\sigma}  |\phi_{n}(x_{n}) |\geq \frac{1}{2}.
\end{equation}
Then, we again have to distinguish two possibilities. Along a subsequence, it follows that $x_{n}\rightarrow x_{0}\in\mathbb{R}^{N}$ or $|x_{n}|\rightarrow +\infty.$

If $x_{n}\rightarrow x_{0},$ standard elliptic estimates show that $\phi_{n}\rightarrow \phi$ uniformly on compact sets of $\mathbb{R}^{N}.$ Moreover, $\phi$ is a solution to (\ref{fast-linear-operaror-1}) with $h^{(2)}=0$ satisfying
\begin{equation}\label{x-j-norm-limit}
<x_{0}-\xi>^{\sigma} |\phi(x_{0})|\geq \frac{1}{2}
\end{equation}
and $|\phi (x)|<+\infty.$
Thus Lemma \ref{orthogonality} shows that
$$\phi=\phi_{1}\in\mbox{Span}\left\{Z_{1},Z_{2},\cdots,Z_{N}\right\}.$$ Then the facts $\int_{\mathbb{R}^{N}}\nabla \phi\cdot \nabla Z_{i}=0$ for $i=1,2,\cdots,N$ show that $\nabla\phi=0.$ We achieve a contradiction to (\ref{x-j-norm-limit}) since $\underset{|x|\rightarrow +\infty}{\lim}\phi(x)=0.$

If $x_{n}\rightarrow +\infty,$ We consider $\tilde{\phi}_{n}(y)=|x_{n}|^{\sigma}\phi_{n}(|x_{n}|y+x_{n}+\xi)$ and observe that $\tilde{\phi}_{n}$ satisfies
$$
\triangle \tilde{\phi}_{n}+|x_{n}|^{2}f'(v_{f,n})\tilde{\phi}_{n}=\tilde{h}^{(2)}_{n},\ x\in\mathbb{R}^{N},
$$
where $v_{f,n}(y)=v_{f}(|x_{n}|y+x_{n}+\xi)$ and $\tilde{h}^{(2)}_{n}(y)=|x_{n}|^{2+\sigma}h_{n}^{(2)}(|x_{n}|y+x_{n}+\xi).$
Noticing that $\|\phi_{n}\|_{*,\xi}=1,$ we have
\begin{equation}\label{fast-lemma-1-phi-1}
|\tilde{\phi}_{n}(y)|\leq \frac{1}{(y+\hat{x}_{n})^{\sigma}},\ \forall y\in\mathbb{R}^{N}\setminus \left\{-\hat{x}_{n}\right\},
\end{equation}
where $\hat{x}_{n}:=\frac{x_{n}}{|x_{n}|}.$
So $\tilde{\phi}_{n}$ is uniformly bounded on compact sets of $\mathbb{R}^{N}\setminus \left\{-\hat{x}_{n}\right\}.$ Similarly, considering that
$$
|\tilde{h}^{(2)}_{n}(y)|\leq \frac{1}{(y+\hat{x}_{n})^{\sigma}}\|h^{(2)}_{n}\|_{**,\xi},\ \forall y\in\mathbb{R}^{N}\setminus \left\{-\hat{x}_{n}\right\},
$$
we obtain $\tilde{h}^{(2)}_{n}\rightarrow 0$ uniformly on compact sets of $\mathbb{R}^{N}\setminus \left\{-\hat{x}_{n}\right\}$ as $n\rightarrow +\infty.$ Thus, by elliptic estimates, we have $\tilde{\phi}_{n}\rightarrow \tilde{\phi}$ uniformly on compact sets of $\mathbb{R}^{N}\setminus \left\{\hat{e}\right\}$ and $\tilde{\phi}$ satisfies
$$
\begin{cases}
\triangle \tilde{\phi}=0, \ x\in\mathbb{R}^{N}\setminus\left\{\hat{e}\right\};\\
|\tilde{\phi}(y)|\leq \frac{1}{|y-\hat{e}|^{\sigma}},\ \forall y\in\mathbb{R}^{N}\setminus \left\{\hat{e}\right\},
\end{cases}
$$
where $\hat{e}=-\underset{n\rightarrow +\infty}{\lim}\frac{x_{n}}{|x_{n}|}.$
By the maximum principle, we conclude that $\tilde{\phi}=0$ which is impossible since $\tilde{\phi}\left(-\hat{e}\right)\neq 0.$

{\bf Step 2. Existence}

We first want to solve (\ref{fast-linear-operator}) on a bounded domain $B_{R}(\xi).$ Let us consider the subspace
$$
X=\left\{\phi \in D_{0}^{1,2}(B_{R}(\xi))\ \mbox{and}\ \int_{B_{R}(\xi)}\psi f'(v_{f})Z_{i}=0,\ i=1,2,\cdots,N \right\}.
$$
Then, according the arguments in \cite{Wei-2016}, finding solution to (\ref{fast-linear-operator}) in this case is equivalent to finding $\phi\in X$ such that
\begin{equation}\label{Fredholm-1}
\int_{B_{R}(\xi)}\nabla \phi \nabla \psi-\int_{B_{R}(\xi)}f'(v_{f})\phi\psi+\int_{B_{R}(\xi)}h \psi=0\ \mbox{for\ all}\ \psi\in X.
\end{equation}
Now, for $h$ satisfying $\|h\|_{**,\xi}<+\infty,$ let us denote by $\phi=A(h)$ the unique solution of the problem
$$
\int_{B_{R}(\xi)}\nabla \phi\nabla \psi +\int_{B_{R}{(\xi)}}h \psi=0\ \mbox{for\ all}\ \psi\in X.
$$
Thus, (\ref{Fredholm-1}) can be written as
$$
\phi-A(f'(v_{f})\phi)=A(h)\ \mbox{for}\ \phi\in X
$$
and, by the compactness of Sobolev's embedding, the map $\phi\rightarrow f'(v_{f})\phi$ is compact.

Hence, we conclude the existence of the solution by the Fredholm alternative since the priori estimate (\ref{fast-linear-norm-up}) implies that the only solution of this equation is $\phi=0$ when $h=0.$ Finally, thanks to the priori estimate again, we can let $R\rightarrow +\infty$ and obtain the existence in the whole space.
\end{proof}

Now we begin to prove Theorem \ref{result-2}. 
We look for a solution of the form $v=v_{f}+\phi$ to the equation (\ref{WZQ-equation-change}) and thus acieve the following equation for $\phi$
\begin{equation}\label{fast-solve-linear-operator-1}
\begin{cases}\triangle \phi+f'(v_{f})\phi=E(v_{f})+F(\phi)+M(\phi),\ x\in\mathbb{R}^{N};\\
\underset{|x|\rightarrow +\infty}{\lim}\phi(x)=0,
\end{cases}
\end{equation}
where
$$
E(v_{f})=\varepsilon V(x-\xi) \frac{G^{-1}(v_{f})}{g(G^{-1}(v_{f}))},
$$
$$
F(\phi)=f(v_{f})+f'(v_{f})\phi-f(v_{f}+\phi)
$$
and
$$
M(\phi)=\varepsilon V(x-\xi)\left[ \frac{G^{-1}(v_{f}+\phi)}{g(G^{-1}(v_{f}+\phi))}-\frac{G^{-1}(v_{f})}{g(G^{-1}(v_{f}))}\right].
$$
However, the problem (\ref{fast-solve-linear-operator-1}) may not be solvable under our situation unless $\xi$ can be chosen in a very special way. So instead of solving (\ref{fast-solve-linear-operator-1}), we consider the following projected problem
\begin{equation}\label{fast-solve-linear-operator-2}
\begin{cases}\triangle \phi+f'(v_{f})\phi=E(v_{f})+F(\phi)+M(\phi)+\sum_{i=1}^{N}c_{i}Z_{i},\ x\in\mathbb{R}^{N};\\
\underset{|x|\rightarrow +\infty}{\lim}\phi(x)=0,
\end{cases}
\end{equation}
where $c_{i}$ are constants.

For $\frac{2}{p-1}<\sigma<N-2,$ we first estimate the error $\|E(v_{f})\|_{**,\xi}$ of the approximate solution $v_{f}.$
Considering that
$$
\left |\frac{G^{-1}(v_{f})}{g(G^{-1}(v_{f}))}\right |\leq v_{f}
$$
and
$$
|v_{f}|\leq C(1+|x|)^{2-N}\ \mbox{for\ all}\ x\in\mathbb{R}^{N},
$$
we have
\begin{align*}
\|E(v_{f})\|_{**,\xi}&=\underset{x\in\mathbb{R}^{N}}{\sup}<x-\xi>^{2+\sigma} |E(v_{f})|\\&
\leq \varepsilon \underset{x\in\mathbb{R}^{N}}{\sup}<x-\xi>^{2+\sigma} V(x-\xi) |v_{f}|\\&
\leq C\varepsilon \underset{x\in\mathbb{R}^{N} }{\sup} \left(\frac{<x-\xi>}{1+|x|}\right)^{\sigma}(1+|x|)^{2-N+\sigma}\\&
\leq C\varepsilon.
\end{align*}

In what follows, by applying the Banach fixed point theorem,
 we can prove that (\ref{fast-solve-linear-operator-2}) is indeed solvable and achieve a solution $(\phi_{\varepsilon},c_{1}, \cdots ,c_{N}).$ We then obtain a solution of the problem (\ref{fast-solve-linear-operator-1}) if $c_{i}=0$ for all $i=1,2,\cdots,N.$


Based on the description of Lemma \ref{fast-lemma-linear-1}, solving (\ref{fast-solve-linear-operator-2}) reduces now to a fixed point problem. Namely, we need to find a fixed point for the map $$(\phi,c_{1},c_{2},\cdots,c_{N})=\mathcal{A}(\phi,c_{1},c_{2},\cdots,c_{N}):=\mathcal{T}(N_{1}(\phi)+N_{2}(\phi)).$$
Here, we will restrict $\phi$ to be small enough such that the function $v_{f}+\phi$ is always positive
and we define the set
$$
\Theta =\left\{(\phi,c_{1},c_{2},\cdots,c_{N})\in  \mathbb{R}^{N+1}\  \big |\  \| \phi \|_{*,\xi} +\sum_{i=1}^{N}|c_{1}|\leq  C\varepsilon\right\}.
$$
We now prove that $\mathcal{A}$ has a fixed point in $\Theta.$

For any $(\phi,c_{1},c_{2},\cdots,c_{N} )\in\Theta,$ we first estimate $M(\phi).$
Note that
$$
\left(\frac{G^{-1}(s)}{g(G^{-1}(s))}\right)'=\frac{1}{g^4(G^{-1}(s))}\leq 1\ \mbox{for\ all}\ s\geq 0.
$$
We have
$$
\left |\frac{G^{-1}(v_{f}+\phi)}{g(G^{-1}(v_{f}+\phi))}-\frac{G^{-1}(v_{f})}{g(G^{-1}(v_{f}))} \right |\leq |\phi|
$$
and then
\begin{equation}\label{M-estimate}
\begin{split}
\|M(\phi)\|_{**,\xi}&=\underset{x\in\mathbb{R}^{N}}{\sup}<x-\xi>^{2+\sigma} |M(\phi)|\\&
\leq \varepsilon \underset{x\in\mathbb{R}^{N}}{\sup}<x-\xi>^{2+\sigma} V(x-\xi) |\phi|\\&
\leq C\varepsilon \|\phi\|_{*,\xi}.
\end{split}
\end{equation}

To estimate $F(\phi),$ we need the following fact: if $1<p<2,$ then
\begin{equation}\label{3-p-1-1}
|f'(s_{1})-f'(s_{2})| \leq C |s_{1}-s_{2}|^{p-1}\ \mbox{for\ all}\ s_{1},s_{2}\geq 0\ \mbox{and}\ |s_{1}-s_{2}|\leq 1.
\end{equation}
Indeed, since
$$
f'(s)=(p-1)\frac{G^{-1}(s)^{p-1}}{g^2(G^{-1}(s))}+\frac{G^{-1}(s)^{p-1}}{g^4(G^{-1}(s))},
$$
we have
\begin{equation}\label{3-p-1-f-0}
\begin{split}
|f'(s_{1})-f'(s_{2})|&\leq (p-1) \left |\frac{G^{-1}(s_{1})^{p-1}}{g^2(G^{-1}(s_{1}))} -\frac{G^{-1}(s_{2})^{p-1}}{g^2(G^{-1}(s_{2}))}\right | \\&+  \left |\frac{G^{-1}(s_{1})^{p-1}}{g^4(G^{-1}(s_{1}))} -\frac{G^{-1}(s_{2})^{p-1}}{g^4(G^{-1}(s_{2}))}\right |.
\end{split}
\end{equation}
Then, noticing that $|s_{1}-s_{2}|\leq 1,$ we have
\begin{equation}\label{3-p-1-f-1}
\begin{split}
&\left |\frac{G^{-1}(s_{1})^{p-1}}{g^2(G^{-1}(s_{1}))} -\frac{G^{-1}(s_{2})^{p-1}}{g^2(G^{-1}(s_{2}))}\right |\\&=
\frac{\left | (1+2G^{-1}(s_{2})^2)G^{-1}(s_{1})^{p-1}- (1+2G^{-1}(s_{1})^2)G^{-1}(s_{2})^{p-1} \right |}{g^2(G^{-1}(s_{1}))g^2(G^{-1}(s_{2}))} \\&\leq
\frac{|G^{-1}(s_{1})^{p-1}-G^{-1}(s_{2})^{p-1}|}{g^2(G^{-1}(s_{1}))g^2(G^{-1}(s_{2}))}+
2\frac{\left | G^{-1}(s_{2})^2G^{-1}(s_{1})^{p-1}- G^{-1}(s_{1})^2G^{-1}(s_{2})^{p-1} \right |}{g^2(G^{-1}(s_{1}))g^2(G^{-1}(s_{2}))} \\&\leq
C |G^{-1}(s_{1})-G^{-1}(s_{2})|^{p-1}+\frac{2G^{-1}(s_{1})^{p-1}G^{-1}(s_{2})^{p-1}}{g^2(G^{-1}(s_{1}))g^2(G^{-1}(s_{2}))} |G^{-1}(s_{2})^{3-p}-G^{-1}(s_{1})^{3-p}|\\&
\leq C |s_{1}-s_{2}|^{p-1}+ \frac{2(3-p)G^{-1}(s_{12})^{2-p}G^{-1}(s_{1})^{p-1}G^{-1}(s_{2})^{p-1}}{g(G^{-1}(s_{12}))g^2(G^{-1}(s_{1}))g^2(G^{-1}(s_{2}))} |s_{1}-s_{2}|\\&\leq
C |s_{1}-s_{2}|^{p-1}
\end{split}
\end{equation}
since
$$
\frac{2(3-p)G^{-1}(s_{12})^{2-p}G^{-1}(s_{1})^{p-1}G^{-1}(s_{2})^{p-1}}{g(G^{-1}(s_{12}))g^2(G^{-1}(s_{1}))g^2(G^{-1}(s_{2}))}\leq C\ \mbox{for\ all}\ s_{1},s_{2}\geq 0,
$$
where $s_{12}$ belongs to the segment jointing $s_{1}$ and $s_{2}.$ On the other hand, by a similar strategy as the proof of the inequality (\ref{3-p-1-f-1}), we conclude that
$$
  \left |\frac{G^{-1}(s_{1})^{p-1}}{g^4(G^{-1}(s_{1}))} -\frac{G^{-1}(s_{2})^{p-1}}{g^4(G^{-1}(s_{2}))}\right |\leq C |s_{1}-s_{2}|^{p-1}
$$
and thus show the inequality (\ref{3-p-1-1}).

Since $\phi$ is small, based on the fact (\ref{3-p-1-1}), we observe that
\begin{equation}\label{fast-N-2}
\begin{split}
|F(\phi)|&=|f(v_{f})+f'(v_{f})\phi-f(v_{f}+\phi)|\\&
\leq  |f'(v_{1})\phi-f'(v_{f})\phi |\\&
\leq \begin{cases}|f''(v_{2})(v_{1}-v_{f})| |\phi|,\ \mbox{if}\ p\geq 2;\\
C |(v_{1}-v_{f})^{p-1} \phi |,\ \mbox{if}\ 1<p<2
\end{cases}\\&
\leq \begin{cases} C |v_{2}|^{p-2} |\phi|^2,\ \mbox{if}\ p\geq 2 ;\\
C |\phi|^{p} ,\ \mbox{if}\ 1<p<2,
\end{cases}
\end{split}
\end{equation}
where $v_{1}$ lies in the segment jointing $v_{f},$ $v_{f}+\phi$ and $v_{2}=tv_{f}+(1-t)v_{1}$ with $t\in [0,1].$
Thus,
jointly with the fact $|v_{2}|\leq C (1+|x|)^{-\sigma}$ for all $x\in\mathbb{R}^{N},$ we have
\begin{equation}\label{F-estimate}
\begin{split}
\|F(\phi)\|_{**,\xi}&=\underset{x\in\mathbb{R}^{N}}{\sup}<x-\xi>^{2+\sigma} |F(\phi)|\\&
\leq \begin{cases} C\underset{x\in\mathbb{R}^{N}}{\sup}<x-\xi>^{2+\sigma} |v_{2}|^{p-2} |\phi|^2,\ \mbox{if}\ p\geq 2 ;\\
C\underset{x\in\mathbb{R}^{N}}{\sup}<x-\xi>^{2+\sigma} |\phi|^{p} ,\ \mbox{if}\ 1<p<2,
\end{cases}\\&
\leq \begin{cases} C\underset{x\in\mathbb{R}^{N}}{\sup}<x-\xi>^{2-(p-1)\sigma} \left(\frac{<x-\xi>}{1+|x|}\right)^{(p-2)\sigma} \|\phi\|_{*,\xi}^2,\ \mbox{if}\ p\geq 2 ;\\
C\underset{x\in\mathbb{R}^{N}}{\sup}<x-\xi>^{2-(p-1)\sigma} \|\phi\|_{*,\xi}^{p} ,\ \mbox{if}\ 1<p<2,
\end{cases}\\&
\leq C\|\phi\|_{*,\xi}^\gamma,
\end{split}
\end{equation}
where $\gamma=\min\left\{2,p\right\}.$

Therefore, by (\ref{M-estimate}) and (\ref{F-estimate}), jointly with Lemma \ref{fast-lemma-linear-1}, it follows that
\begin{equation}\label{A-estimate-1}
\begin{split}
\|\mathcal{A}(\phi,c_{1},c_{2},\cdots,c_{N} )\|_{*,\xi}\leq &C(\|E(v_{f})\|_{**,\xi}+\|M(\phi)\|_{**,\xi}+\|F(\phi)\|_{**,\xi})\\&
\leq C(\varepsilon+\varepsilon\|\phi\|_{*,\xi}+\|\phi\|_{*,\xi}^\gamma)\\&
\leq C\varepsilon,
\end{split}
\end{equation}
which shows $\mathcal{A}(\Theta)\subset \Theta.$

We still have to prove that $\mathcal{A}$ is a contraction mapping in $\Theta.$ If we take $$(\phi_{1},c_{1,1},c_{2,1},\cdots,c_{N,1} ), (\phi_{2},c_{1,2},c_{2,2},\cdots,c_{N,2} )\in\Theta,$$ then we have
\begin{equation}\label{fast-mapp-1}
\begin{split}
&\|\mathcal{A}(\phi_{1},c_{1,1},c_{2,1},\cdots,c_{N,1})-\mathcal{A}( \phi_{2},c_{1,2},c_{2,2},\cdots,c_{N,2} )\|_{*,\xi}\\&\leq C[\|M(\phi_{1})-M(\phi_{2})\|_{**,\xi}+\|F(\phi_{1})-F(\phi_{2})\|_{**,\xi}].
\end{split}
\end{equation}
To estimate $\|M(\phi_{1})-M(\phi_{2})\|_{**,\xi},$ we note that
\begin{equation}\label{fast-mapp-N-1-1}
\begin{split}
| M(\phi_{1})-M(\phi_{2})|=|D_{\phi}M(\bar{\phi})(\phi_{1}-\phi_{2})|,
\end{split}
\end{equation}
where $\bar{\phi}$ lies in the segment joining $\phi_{1}$ and $\phi_{2}.$
Moreover, a direct calculation shows
\begin{equation}\label{mapp-N-1-2}
\begin{split}
|D_{\phi}M(\phi)|&=\left | \frac{\varepsilon V(x-\xi)}{g^2(G^{-1}(v_{f}+\phi))}\left[1-\frac{2G^{-1}(v_{f}+\phi)^2}{g^2(G^{-1}(v_{f}+\phi))}\right] \right |\\&
= \left | \frac{\varepsilon V(x-\xi)}{g^4(G^{-1}(v_{f}+\phi))}\right |\leq \varepsilon V(x-\xi).
\end{split}
\end{equation}
Then,
\begin{equation}\label{mapp-N-1-3}
\begin{split}
\underset{x\in\mathbb{R}^{N}}{\sup}<x-\xi>^{2+\sigma}|M(\phi_{1})-M(\phi_{2})|& \leq \underset{x\in\mathbb{R}^{N}}{\sup} |x-\xi |^{2} |D_{\phi}N_{1}(\overline{\phi})| \|\phi_{1}-\phi_{2}\|_{*,\xi}\\&
\leq C \varepsilon \|\phi_{1}-\phi_{2}\|_{*,\xi} \underset{x\in\mathbb{R}^{N}}{\sup} <x-\xi>^{2}V(x-\xi)\\&
\leq C\varepsilon \|\phi_{1}-\phi_{2}\|_{*,\xi}.
\end{split}
\end{equation}
Thus, we have
\begin{equation}\label{M-Con}
\|M(\phi_{1})-M(\phi_{2})\|_{**,\xi}\leq C\varepsilon \|\phi_{1}-\phi_{2}\|_{*,\xi}.
\end{equation}

Now we estimate $\|F(\phi_{1})-F(\phi_{2})\|_{**,\xi}.$ We note that
\begin{equation}\label{fast-mapp-N-2-1}
\begin{split}
|F(\phi_{1})-F(\phi_{2})|=|D_{\phi}N_{2}(\bar{\phi})(\phi_{1}-\phi_{2})|,
\end{split}
\end{equation}
where $\bar{\phi}$ lies in the segment joining $\phi_{1}$ and $\phi_{2}.$
Moreover,
\begin{equation}\label{mapp-N-2-3}
\begin{split}
|D_{\phi}N_{2}(\phi)|&=|f'(v_{f})-f'(v_{f}+\phi)|\\&
\leq \begin{cases} C |v_{1}|^{p-2} |\phi|,\ \mbox{if}\ p\geq 2 ;\\
C |\phi|^{p-1} ,\ \mbox{if}\ 1<p<2,
\end{cases}
\end{split}
\end{equation}
where $v_{1}=tv_{f}+(1-t)(v_{f}+\phi)$ with $t\in [0,1].$
Then, similarly as the proof of (\ref{F-estimate}), we have
\begin{equation}\label{F-Con}
\|F(\phi_{1})-F(\phi_{2})\|_{**,\xi}\leq C \|\bar{\phi}\|^{\min\left\{1,p-1\right\}}_{*,\xi} \|\phi_{1}-\phi_{2}\|_{*,\xi}\leq C\varepsilon^{\min\left\{1,p-1\right\}} \|\phi_{1}-\phi_{2}\|_{*,\xi}.
\end{equation}
Thus, under our situation, combining (\ref{fast-mapp-1}), (\ref{M-Con}) and (\ref{F-Con}), we have
that $\mathcal{A}$ is a contraction mapping in $\Theta$ and hence there indeed exists a fixed point $(\phi_{\varepsilon},c_{1},c_{2},\cdots,c_{N}).$

In what follows of this section, we will complete the proof of Theorem \ref{result-2}.

We have found a solution $(\phi_{\varepsilon},c_{1},c_{2},\cdots,c_{N})$ to (\ref{fast-solve-linear-operator-2}) satisfying
$$
\|\phi_{\varepsilon} \|_{*,\xi}+\sum_{i=1}^{N}|c_{i}| \leq C\varepsilon.
$$
 To prove the result contained in Theorem \ref{result-2}, it suffices to show that the point $\xi$
 can be adjust so that the constants $c_{1},c_{2},\cdots,c_{N}$ are all contemporarily equal to zero. Combining Lemma \ref{fast-lemma-linear-1},
 we only need to show
\begin{equation}\label{fast-Z-1-N-condition}
\int_{\mathbb{R}^{N}}(E(v_{f})+M(\phi_{\varepsilon})+F(\phi_{\varepsilon}))\frac{\partial v_{f}}{\partial x_{j}}=0\ \mbox{for}\ j=1,2,\cdots,N.
\end{equation}

We first define
\begin{equation}\label{fast-Z-1-function-1}
\begin{split}
F_{j}(\varepsilon,\xi)&=\int_{\mathbb{R}^{N}} (E(v_{f})+M(\phi_{\varepsilon})+F(\phi_{\varepsilon})) \frac{\partial v_{0}}{\partial x_{j}}.
\end{split}
\end{equation}
The subordinate terms in (\ref{fast-Z-1-function-1}) are $\int_{\mathbb{R}^{N}}F(\phi_{\varepsilon})\frac{\partial v_{f}}{\partial y_{j}} $ and $\int_{\mathbb{R}^{N}}M(\phi_{\varepsilon})\frac{\partial v_{f}}{\partial y_{j}} .$ Indeed,
we have the following estimates
\begin{equation}\label{F-subordinate}
\begin{split}
\left | \int_{\mathbb{R}^{N}}F(\phi_{\varepsilon})\frac{\partial v_{f}}{\partial y_{j}}\right |\leq \|F(\phi_{\varepsilon})\|_{**,\xi}\int_{\mathbb{R}^{N}}<x-\xi>^{-2-\sigma}\left |\frac{\partial v_{f}}{\partial y_{j}}\right |=O(\varepsilon^\gamma)
\end{split}
\end{equation}
and
\begin{equation}\label{M-subordinate}
\begin{split}
\left | \int_{\mathbb{R}^{N}}M(\phi_{\varepsilon})\frac{\partial v_{f}}{\partial y_{j}}\right |\leq \|M(\phi_{\varepsilon})\|_{**,\xi}\int_{\mathbb{R}^{N}}<x-\xi>^{-2-\sigma}\left |\frac{\partial v_{f}}{\partial y_{j}}\right |=O(\varepsilon^2).
\end{split}
\end{equation}

Noticing that there exist $k_{0}\in (0,1)$ and $k_{1}>1$ such that $k_{0}v_{f}\leq \frac{G^{-1}(v_{f})}{g(G^{-1}(v_{f}))}\leq k_{1}v_{f},$ the dominant term in (\ref{fast-Z-1-function-1}) satisfies
\begin{equation}\label{fast-dominant-term}
\begin{split}
\int_{\mathbb{R}^{N}}E(v_{f})\frac{\partial v_{f}}{\partial y_{j}}=\varepsilon\int_{\mathbb{R}^{N}}V(y-\xi)\frac{G^{-1}(v_{f})}{g(G^{-1}(v_{f}))}\frac{\partial v_{f}}{\partial y_{j}}\sim &
\varepsilon\int_{\mathbb{R}^{N}}V(y-\xi)v_{f}\frac{\partial v_{f}}{\partial y_{j}} \\
=&\varepsilon\int_{\mathbb{R}^{N}}V(y-\xi)v_{f}\frac{\partial v_{f}}{\partial y_{j}}\\=&
\frac{\varepsilon}{2}\int_{\mathbb{R}^{N}}v_{f}^{2}\frac{\partial V}{\partial \xi_{j}}(y-\xi)
\\=&\frac{\varepsilon}{2}\frac{\partial}{\partial \xi_{j}}\int_{\mathbb{R}^{N}}v_{f}^{2}V(y-\xi).
\end{split}
\end{equation}

Combining (\ref{F-subordinate}), (\ref{M-subordinate}) and (\ref{fast-dominant-term}), we achieve
$$
F_{j}(\varepsilon,\xi)\sim\frac{\varepsilon}{2}\frac{\partial}{\partial \xi_{j}}\int_{\mathbb{R}^{N}}v_{f}^2 V(y-\xi)+o(\varepsilon)\ \mbox{for}\ j=1,2,\cdots,N.
$$
Thus, if we set $G(\xi)=\int_{\mathbb{R}^{N}}v_{f}^2V(y-\xi),$ then $G(0)>0$ and $\underset{|x|\rightarrow +\infty}{\lim}G(\xi)=0.$ This implies that $G$ attains a global maximum point $\xi_{0}\in B_{M}(0)$ for some $M>0.$ By the definition of stable critical point (Musso and Pistoia \cite{Musso-2002}), $G$ has a stable critical point in $B_{M}(0)$ and as a result, we deduce that, for $\varepsilon$ small, $F(\varepsilon,\xi)=(F_{1}(\varepsilon,\xi),F_{2}(\varepsilon,\xi),\cdots, F_{N}(\varepsilon,\xi))$ has a zero point in $B_{M}(0).$ Consequently, $c_{j}=0$ for $j=1,2,\cdots,N.$

\section{Proof of Theorem \ref{result-3}}

In this section, we will construct slow decay solutions to the problem (\ref{WZQ-equation}) with $\varepsilon=1.$
The results of Theorem \ref{result-3} are based on a suitable linear theory devised for the linearized operator associated
to the equation (\ref{WZQ-equation}) at $u=w$ in the entire space $\mathbb{R}^{N}$ and in the application of perturbation arguments. We consider $w$ as an approximation for a solution of (\ref{WZQ-equation}), provided that $\lambda>0$ is chosen small enough. To this aim, we need to know the solvability of the operator $\triangle -V_{\lambda}+pw^{p-1}$ in suitable weighted Sobolev space.

Let
$$Z_{i}=\eta \frac{\partial w}{\partial x_{i}}, i=1,2,\cdots,N,$$
where $\eta\in C_{0}^{\infty}(\mathbb{R}^{N})$ satisfies $0\leq \eta\leq 1.$ Moreover, $\eta(x)=1$ if $|x|\leq R_{0}$ and $\eta(x)=0$ if $|x|\geq R_{0}+1$ for a fixed number $R_{0}>0$ large enough.

Under appropriate norms
\begin{equation}\label{4-norm-1}
\|\phi\|_{*,\xi}=\underset{|x-\xi|\leq 1}{\sup}|x-\xi|^{\sigma}|\phi(x)|+\underset{|x-\xi|\geq 1}{\sup}|x-\xi|^{\frac{2}{p-1}}|\phi(x)|
\end{equation}
and
\begin{equation}\label{4-norm-2}
\|h\|_{**,\xi}=\underset{|x-\xi|\leq 1}{\sup}|x-\xi|^{2+\sigma}|h(x)|+\underset{|x-\xi|\geq 1}{\sup}|x-\xi|^{2+\frac{2}{p-1}}|h(x)|,
\end{equation}
where $\sigma>0$ and $\xi\in\mathbb{R}^{N},$ we first consider the solvability of the linear problem
\begin{equation}\label{4-fast-linear-operator}
\begin{cases}\triangle \phi-V_{\lambda}(x)\phi+pw^{p-1}\phi=h+\sum_{i=1}^{N}c_{i}Z_{i},\ x\in\mathbb{R}^{N};\\
\underset{|x|\rightarrow +\infty}{\lim}\phi(x)=0
\end{cases}
\end{equation}
and thus we need the following lemma which is proved by D\'avila, del Pino, Musso and Wei
in \cite{Pino-Wei-2007}.
\begin{lemma}\label{4-slow-linear}
Let $|\xi|\leq \Lambda.$ Suppose $V$ satisfies (\ref{V-condition}) and $\|h\|_{**,\xi}<\infty.$ Then, for $\lambda>0$ sufficiently small,
\begin{itemize}
\item[{\rm (1).}] if $N\geq 4,$ $p>\frac{N+1}{N-3},$ equation (\ref{4-fast-linear-operator}) with $c_{i}=0$ for $1\leq i\leq N$ and $\xi=0$ has a solution $\phi=\mathcal{T}_{\lambda}(h)$ which depends linearly on $h$ and there exist a constant $C$ independent with $\lambda$ such that
$$
\|\mathcal{T}_{\lambda}(h)\|_{*,0}\leq C\|h\|_{**,0};
$$
\item[{\rm (2).}] if $N\geq 3,$ $\frac{N+2}{N-2}<p< \frac{N+1}{N-3}$ and $V$ also satisfies (\ref{slow-V-condition-2}), equation (\ref{4-fast-linear-operator}) has a solution $(\phi,c_{1},c_{2},\cdots,c_{N})=\mathcal{T}_{\lambda}(h)$ which depends linearly on $h$ and there exist a constant $C$ independent with $\lambda$ such that
$$
\|\phi\|_{*,\xi}+\underset{1\leq i\leq N}{\max}|c_{i}|\leq C\|h\|_{**,\xi}.
$$
Moreover, $c_{i}=0$ for all $1\leq i\leq N$ if and only if
\begin{equation}\label{3-slow-or}
\int_{\mathbb{R}^{N}}h\frac{\partial w}{\partial x_{i}}=0\ \mbox{for}\ 1\leq N\leq N.
\end{equation}
\end{itemize}
\end{lemma}

Based on Lemma \ref{4-slow-linear}, we can prove Theorem \ref{result-3}. We look for a solution of the form $v=w+\phi$ to the equation (\ref{WZQ-equation-change-lambda}) and, for $S(w),$ $N(\phi)$ defined in Section 2 and $h(s):=\frac{G^{-1}(s)}{g(G^{-1}(s))}$, we achieve the following equation
\begin{equation}\label{2-phi-solution}
\begin{cases}
\triangle \phi-V_{\lambda}(x)\phi+pw^{p-1}\phi=S_{1}(w)+N(\phi)+P(\phi),\ x\in \mathbb{R}^{N};\\
\underset{|x|\rightarrow +\infty}{\lim}\phi(x)=0,
\end{cases}
\end{equation}
where
$$
S_{1}(w)=S(w)+V_{\lambda}(x)\lambda^{-\frac{2}{p-1} } h(\lambda^{\frac{2}{p-1}}w)
$$
and
$$
P(\phi)=V_{\lambda}(x) [\lambda^{-\frac{2}{p-1} } (h(\lambda^{\frac{2}{p-1}}(w+\phi))-h(\lambda^{\frac{2}{p-1}}w))-\phi].
$$

{\bf The case $p>\frac{N+1}{N-3}$}

In this case, we rescale $v(x)$ as $\lambda^{\frac{2}{p-1}}v(\lambda x),$ that is, $\xi=0$ in the previous paragraph.
Computations show that
\begin{align*}
\lambda^{-\frac{2}{p-1} } h(\lambda^{\frac{2}{p-1}}w) =w+o(1)\ \mbox{as}\ \lambda\rightarrow 0.
\end{align*}
According to the arguments in \cite{Pino-Wei-2007}, we know
$$
\|V_{\lambda}(x) \lambda^{-\frac{2}{p-1} } h(\lambda^{\frac{2}{p-1}}w) \|_{**,0}:=\delta(\lambda)=o(1)\ \mbox{as}\ \lambda\rightarrow 0.
$$
Thus, for $\sigma\in \left(0,\min\left\{2,\frac{2}{p-1}\right\}\right),$ the error of the approximate solution in the norm (\ref{4-norm-2}) is
\begin{equation}\label{4-error}
\|S_{1}(w)\|_{**,0}=\|S(w)+V_{\lambda}(x) \lambda^{-\frac{2}{p-1} } h(\lambda^{\frac{2}{p-1}}w)\|_{**,0}\leq C\rho(\lambda),
\end{equation}
where $\rho(\lambda):=\lambda^{\frac{4}{p-1}}+\delta(\lambda).$
Consequently, for the operator $\mathcal{A}_{\lambda}(\phi):=\mathcal{T}_{\lambda}(S_{1}(w)+N(\phi)+P(\phi)),$ where $\mathcal{T}_{\lambda}$ is given in Lemma \ref{4-slow-linear}-(1), we can use the contraction mapping theorem on
$$
\Sigma_{\lambda}=\left\{\phi: \mathbb{R}^{N}\rightarrow \mathbb{R} \big |\  \|\phi\|_{*,0}\leq C\rho(\lambda) \right\}
$$
and we will prove that $\mathcal{A}_{\lambda}$ has a fixed point in $\Sigma_{\lambda}.$

For any $\phi\in \Sigma_{\lambda},$ we first give the estimate of $\|P(\phi)\|_{**,0}.$ We observe that, for a number $C_{3}>0,$
\begin{equation}\label{add-P-1}
\begin{split}
\lambda^{-\frac{2}{p-1} } (h(\lambda^{\frac{2}{p-1}}(w+\phi))-h(\lambda^{\frac{2}{p-1}}w))-\phi
&=
-C_{3}\lambda^{\frac{4}{p-1}}[(w+\phi)^{3}-w^3]\\&+o\left(\lambda^{\frac{4}{p-1}}\right)\ \mbox{as}\ \lambda\rightarrow 0.
\end{split}
\end{equation}
Thus, we have
\begin{equation}\label{P-estimate-1}
\begin{split}
\underset{|x|\leq 1}{\sup}|x|^{2+\sigma}|P(\phi)|
\leq &C\underset{|x|\leq 1}{\sup}|x|^{2+\sigma}V_{\lambda}(x)\lambda^{\frac{4}{p-1}}(|w^{2}\phi|+|\phi|^{3})\\
\leq & C\lambda^{\frac{4}{p-1}}\underset{|x|\leq 1}{\sup}|x|^{\sigma}(|w^2\phi|+|\phi|^{3})\\
\leq & C\lambda^{\frac{4}{p-1}}\|\phi\|_{*,0}.
\end{split}
\end{equation}
On the other hand, combining
\begin{equation}\label{add-phi}
|\phi(x)|\leq C|x|^{-\frac{2}{p-1}}\|\phi\|_{*,0}\  \mbox{for}\  |x|\geq 1
\end{equation}
and
\begin{equation}\label{add-w}
w(x)\leq C(1+|x|)^{-\frac{2}{p-1}}\  \mbox{for}\  x\in\mathbb{R}^{N},
\end{equation}
we have
\begin{equation}\label{P-estimate-5}
\begin{split}
\underset{|x|\geq 1}{\sup}|x|^{2+\frac{2}{p-1}}|P(\phi)|
\leq &C\lambda^{\frac{4}{p-1}}\underset{|x|\geq 1}{\sup}|x|^{2+\frac{2}{p-1}}V_{\lambda}(x)( |w^2\phi|+|\phi|^3)\\
\leq & C\lambda^{\frac{4}{p-1}}\underset{|x|\geq 1}{\sup}\left((1+|x|)^{-\frac{4}{p-1}}\|\phi\|_{*,0}+|x|^{-\frac{4}{p-1}}\|\phi\|_{*,0}^3\right)\\
\leq & C\lambda^{\frac{4}{p-1}}(\|\phi\|_{*,0}+\|\phi\|_{*,0}^{3}).
\end{split}
\end{equation}
Consequently, combining (\ref{P-estimate-1}) and (\ref{P-estimate-5}), we have
\begin{equation}\label{P-estimate-6}
\begin{split}
\|P(\phi)\|_{**,0}&\leq C\lambda^{\frac{4}{p-1}}(\|\phi\|_{*,0}+\|\phi\|_{*,0}^{3}).
\end{split}
\end{equation}

Thus, jointly with the estimate of $\|N(\phi)\|_{**,0}$ in Section 2, we conclude
\begin{align*}
\|\mathcal{A}_{\lambda}(\phi)\|_{*,0}& \leq C \|S_{1}(w)+N(\phi)+P(\phi)\|_{**,0}\\&
\leq C[\rho(\lambda)+\rho(\lambda)^2+\rho(\lambda)^p+\rho(\lambda)+\rho(\lambda)^3]\\&
\leq C\rho(\lambda) \ \mbox{for\ any}\ \phi\in \Sigma_{\lambda}.
\end{align*}
That is, $\mathcal{A}_{\lambda}(\Sigma_{\lambda})\subset \Sigma_{\lambda}.$

For any $\phi_{1},\phi_{2}\in \Sigma_{\lambda},$ we want to estimate $\|P(\phi_{1})-P(\phi_{2})\|_{**,0}.$ We note that
\begin{equation}\label{mapp-P-1}
\begin{split}
|P(\phi_{1})-P(\phi_{2})|=|D_{\phi}P(\bar{\phi})(\phi_{1}-\phi_{2})|,
\end{split}
\end{equation}
where $\bar{\phi}$ lies in the segment joining $\phi_{1}$ and $\phi_{2}.$ Then, it follows that
$$
|x|^{2+\sigma}|P(\phi_{1})-P(\phi_{2})|\leq |x|^{2} |D_{\phi}P(\bar{\phi})| \|\phi_{1}-\phi_{2}\|_{*,0}\ \mbox{for}\ |x|\leq 1
$$
and
$$
|x|^{2+\frac{2}{p-1}}|P(\phi_{1})-P(\phi_{2})|\leq |x|^{2} |D_{\phi}P(\bar{\phi})| \|\phi_{1}-\phi_{2}\|_{*,0}\ \mbox{for}\ |x|\geq 1.
$$
Thus, we have
\begin{equation}\label{mapp-P-2}
\|P(\phi_{1})-P(\phi_{2})\|_{**,0}\leq C\underset{x\in\mathbb{R}^{N}}{\sup}\left(|x|^{2}|D_{\phi}P(\bar{\phi})|\right)\|\phi_{1}-\phi_{2}\|_{*,0}.
\end{equation}
Moreover, a direct calculation shows
\begin{equation}\label{mapp-P-3}
\begin{split}
D_{\phi}P(\phi)&=V_{\lambda}(x)\left(\frac{1}{g^2(G^{-1}(\lambda^{\frac{2}{p-1}}(w+\phi)))}-\frac{2(G^{-1}(\lambda^{\frac{2}{p-1}}(w+\phi)))^2}{g^4(G^{-1}(\lambda^{\frac{2}{p-1}}(w+\phi)))}-1\right)\\&
=V_{\lambda}(x)\left[1-4\lambda^{\frac{4}{p-1}}(w+\phi)^2+o(\lambda^{\frac{4}{p-1}}(w+\phi)^2)-1\right]\ \mbox{as}\ \lambda\rightarrow 0.
\end{split}
\end{equation}
To go a step further, based on the arguments in the previous paragraph, we conclude that
  \begin{equation}\label{mapp-P-4}
\begin{split}
\underset{x\in\mathbb{R}^{N}}{\sup}\left(|x|^{2}|D_{\phi}P(\bar{\phi})|\right)&\leq C\lambda^{\frac{4}{p-1}}.
\end{split}
\end{equation}
Consequently, combining (\ref{mapp-P-2}) and (\ref{mapp-P-4}), it follows that
\begin{equation}\label{mapp-P-5}
\|P(\phi_{1})-P(\phi_{2})\|_{**,0}\leq \frac{1}{4}\|\phi_{1}-\phi_{2}\|_{*,0}
\end{equation}
for $\lambda$ sufficiently small.

It is straightforward to show that
\begin{equation}\label{4-Con}
\begin{split}
\|\mathcal{A}_{\lambda}(\phi_{1})-\mathcal{A}_{\lambda}(\phi_{2})\|_{*,0} &\leq C[\|N(\phi_{1})-N(\phi_{2})\|_{**,0}+\|P(\phi_{1})-P(\phi_{2})\|_{**,0}]\\&
\leq \frac{1}{2}\|\phi_{1}-\phi_{2}\|_{*,0}\ \mbox{for}\ \phi_{1},\phi_{2}\in \Sigma_{\lambda}
\end{split}
\end{equation}
since we can achieve that
$$
\|N(\phi_{1})-N(\phi_{2})\|_{**,0}\leq \frac{1}{4}\|\phi_{1}-\phi_{2}\|_{*,0}
$$
according to Section 2 for $\lambda$ sufficiently small. This means that $\mathcal{A}_{\lambda}$ is a contraction mapping from $\Sigma_{\lambda}$ into itself and hence a fixed point $\phi_{\lambda}$ indeed exists. So the function $v_{\lambda}(x):=\lambda^{\frac{2}{p-1}}(w(\lambda x)+\phi_{\lambda}(\lambda x))$ is a continuum solutions of (\ref{WZQ-equation-change-lambda}) satisfying $\underset{\lambda\rightarrow 0}{\lim} v_{\lambda}(x)=0$ uniformly in $\mathbb{R}^{N}$ and $u_{\lambda}(x)=G^{-1}(v_{\lambda}(x))$ is our desired solution to (\ref{WZQ-equation}).

{\bf The case $\frac{N+2}{N-2}<p<\frac{N+1}{N-3}$}

In this case, the problem (\ref{2-phi-solution}) may not be solvable under our situation unless $\xi$ is chosen in a very special way. So, instead of solving (\ref{2-phi-solution}), we consider the following projected problem
\begin{equation}\label{2-phi-solution-or}
\begin{cases}
\triangle \phi-V_{\lambda}(x)\phi+pw^{p-1}\phi=S_{1}(w)+N(\phi)+P(\phi)+\sum_{i=1}^{N}c_{i}Z_{i},\ x\in \mathbb{R}^{N};\\
\underset{|x|\rightarrow +\infty}{\lim}\phi(x)=0,
\end{cases}
\end{equation}
where $c_{i}$ are constants.
Moreover, we will slightly change the previous definition of the norms as
$$
\|\phi\|_{*,\xi}^{(\theta)} =\underset{|x-\xi|\leq 1}{\sup}|x-\xi|^{\theta}|\phi(x)|+\underset{|x-\xi|\geq 1}{\sup}|x-\xi|^{\frac{2}{p-1}}|\phi(x)|
$$
and
$$
\|h\|_{**,\xi}^{(\theta)} =\underset{|x-\xi|\leq 1}{\sup}|x-\xi|^{2+\theta}|h(x)|+\underset{|x-\xi|\geq 1}{\sup}|x-\xi|^{2+\frac{2}{p-1}}|h(x)|.
$$
Just as the case $p>\frac{N+1}{N-3},$ we can prove that (\ref{2-phi-solution-or}) is indeed solvable and achieve a solution $(\phi(\lambda ,\xi),c_{1}(\lambda,\xi),c_{2}(\lambda,\xi)\cdots ,c_{N}(\lambda,\xi)).$ We then obtain a solution of the problem (\ref{2-phi-solution}) if $c_{i}(\lambda,\xi)=0$ for all $i=1,2,\cdots,N.$

Here, we also fix $\sigma\in \left(0,\min\left\{2,\frac{2}{p-1}\right\}\right)$ and find the error of the approximate solution is
\begin{equation}\label{4-error-case-2}
\|S_{1}(w)\|_{**,\xi}^{(\sigma)}\leq C\rho(\lambda),
\end{equation}
where $\rho(\lambda)=o(1)$ as $\lambda\rightarrow 0.$
So we can define
$$
\Sigma_{\lambda,\sigma}=\left\{ (\phi, c_{1},c_{2},\cdots,c_{N}) \in\mathbb{R}^{N+1}\  \big  | \ \|\phi\|_{*,\xi}^{(\sigma)}+\sum_{i=1}^{N}c_{i}\leq \rho(\lambda)\right\}.
$$
Similarly, as the proof of the previous case, jointly with Lemma  \ref{4-slow-linear}-(2),
we conclude that the operator $(\phi, c_{1},c_{2},\cdots,c_{N})=\mathcal{A}_{\lambda}(\phi, c_{1},c_{2},\cdots,c_{N}):=\mathcal{T}_{\lambda}(S_{1}(w)+N(\phi)+P(\phi))$ is a contraction mapping in $\Sigma_{\lambda,\sigma}$ and hence achieve a fixed point
$$(\phi(\lambda ,\xi),c_{1}(\lambda,\xi),c_{2}(\lambda,\xi)\cdots ,c_{N}(\lambda,\xi))\in \Sigma_{\lambda,\sigma},$$ which satisfies the equation (\ref{2-phi-solution-or}).
Moreover, under the condition (\ref{slow-V-condition-2}), we observe that $\rho(\lambda)$ can be taken as $\lambda^{\theta}$ in (\ref{4-error-case-2}) for any $\theta\in \left(0,\frac{4}{p-1} \right).$
That is,
\begin{equation}\label{4-phi-bound}
\begin{split}
\| \lambda^{-\frac{2}{p-1}}V_{\lambda}h(\lambda^{\frac{2}{p-1}}w)\|_{**,\xi}^{(\theta)}\leq C\lambda^{\theta}\ \mbox{for}\ \theta\in \left(0,N-2 \right),
\end{split}
\end{equation}
\begin{equation}\label{4-phi-bound-1}
\begin{split}
\| S(w)\|_{**,\xi}^{(\theta)}\leq C\lambda^{\theta}\ \mbox{for}\ \theta\in \left(0,\frac{4}{p-1} \right)
\end{split}
\end{equation}
and
\begin{equation}\label{phi-bound-3-3}
\begin{split}
\|\phi(\lambda,\xi)\|_{*,\xi}^{(\theta)}+\underset{1\leq i\leq N}{\max}|c_{i}(\lambda)|\leq C\lambda^{\theta}\ \mbox{for}\ \theta\in \left(0,\frac{4}{p-1} \right).
\end{split}
\end{equation}

Thus, to complete our proof, by Lemma \ref{4-slow-linear}-(2) we need to find $\xi=\xi_{\lambda}$ such that
\begin{equation}\label{or-condition-1}
\int_{\mathbb{R}^{N}}\left(S(w)+V_{\lambda}\lambda^{-\frac{2}{p-1}}h(\lambda^{\frac{2}{p-1}}w)+N(\phi(\lambda,\xi))+P(\phi(\lambda,\xi))\right)\frac{\partial w}{\partial x_{j}}=0,\ 1\leq j\leq N.
\end{equation}
Combining the arguments in \cite{Pino-Wei-2007} and noticing that $\frac{4}{p-1}<N-2$, we know
\begin{equation}\label{int-V-estimate}
\int_{\mathbb{R}^{N}}\lambda^{-\frac{2}{p-1}}V_{\lambda}h(\lambda^{\frac{2}{p-1}}w) \frac{\partial w}{\partial x_{j}}=o(\lambda^{\frac{4}{p-1}})\ \mbox{as}\ \lambda\rightarrow 0
\end{equation}
and
\begin{equation}\label{int-N-estimate}
\int_{\mathbb{R}^{N}}N(\phi(\lambda,\xi)) \frac{\partial w}{\partial x_{j}}=o(\lambda^{\frac{4}{p-1}})\ \mbox{as}\ \lambda\rightarrow 0.
\end{equation}
Moreover, noticing that
$$
|P(\phi(\lambda,\xi))|\leq C\lambda^{\frac{4}{p-1}}V_{\lambda}(|w^2\phi(\lambda,\xi)|+|\phi(\lambda,\xi)|^3),
$$
we have
\begin{equation}\label{int-P-estimate}
\begin{split}
\left | \int_{\mathbb{R}^{N}}P(\phi(\lambda,\xi))\frac{\partial w}{\partial x_{j}}\right |& \leq C\lambda^{\frac{4}{p-1}} \left |\int_{\mathbb{R}^{N}}V_{\lambda}\phi(\lambda,\xi) (w^2+\phi^2(\lambda,\xi)) \frac{\partial w}{\partial x_{j}}\right |\\&
\leq C\lambda^{\frac{4}{p-1}} \left |\int_{\mathbb{R}^{N}}V_{\lambda}\phi(\lambda,\xi) \frac{\partial w}{\partial x_{j}}\right |\\&
=o(\lambda^{\frac{4}{p-1}})\ \mbox{as}\ \lambda\rightarrow 0.
\end{split}
\end{equation}
Now, we claim that the dominant term in (\ref{or-condition-1}) is
\begin{equation}\label{dominant-term}
\int_{\mathbb{R}^{N}}S(\phi(\lambda,\xi)) \frac{\partial w}{\partial x_{j}}.
\end{equation}
Note that
$$
S(w)=C_{p}\lambda^{\frac{4}{p-1}}w^{p+2}+o(\lambda^{\frac{4}{p-1}}w^{p+2})\ \mbox{as}\ \lambda\rightarrow 0.
$$
We have
\begin{equation}\label{int-S-estimate}
\begin{split}
\int_{\mathbb{R}^{N}}S(\phi(\lambda,\xi)) \frac{\partial w}{\partial x_{j}}=C_{p}\lambda^{\frac{4}{p-1}}\int_{\mathbb{R}^{N}}
w^{p+2}(x+\xi)\frac{\partial w}{\partial x_{j}}(x+\xi)\mbox{d}x+o(\lambda^{\frac{4}{p-1}})\ \mbox{as}\ \lambda\rightarrow 0.
\end{split}
\end{equation}
If we define
$$
F_{\lambda}^{(j)}(\xi)=\int_{\mathbb{R}^{N}}\left(S(w)+V_{\lambda}\lambda^{-\frac{2}{p-1}}h(\lambda^{\frac{2}{p-1}}w)+N(\phi(\lambda,\xi))+P(\phi(\lambda,\xi))\right)\frac{\partial w}{\partial x_{j}}
$$
and $F_{\lambda}(\xi)=(F_{\lambda}^{(1)}(\xi),F_{\lambda}^{(2)}(\xi),\cdots,F_{\lambda}^{(N)}(\xi)).$ Then, by (\ref{int-V-estimate}), (\ref{int-N-estimate}), (\ref{int-P-estimate}) and (\ref{int-S-estimate}), we achieve that
$$
F_{\lambda}^{(j)}(\xi)=C_{p}\lambda^{\frac{4}{p-1}}\int_{\mathbb{R}^{N}}
w^{p+2}(x+\xi)\frac{\partial w}{\partial x_{j}}(x+\xi)\mbox{d}x+o(\lambda^{\frac{4}{p-1}}) \ \mbox{as}\ \lambda\rightarrow 0
$$
and so we can show the existence of a solution $\xi_{\lambda}$ to (\ref{or-condition-1}) since $0$ is a critical point of $w.$
Thus, we conclude that
$$
\langle F_{\lambda}(\xi), \xi\rangle<0\ \mbox{for}\ |\xi|=\delta,
$$
where $\delta$ is a fixed small constant. Using this fact and degree theory we obtain the existence of $\xi_{\lambda}$ such that $F_{\lambda}(\xi_{\lambda})=0$ in $B_{\delta}.$ This complete the proof of Theorem \ref{result-3}.

\newpage
\end{document}